\documentclass[reqno,a4paper,12pt]{amsart} 
    
    \usepackage{amsmath,amscd,amsfonts,amssymb}
    \usepackage{mathrsfs,dsfont}
    \usepackage{color}
    \usepackage{mathtools}
    \usepackage{tikz-cd}
\usepackage[colorlinks=true,linkcolor=blue,citecolor=blue,hypertexnames=false]{hyperref}

    \usepackage{tikz}
    \usetikzlibrary{decorations.pathreplacing}
    
    \numberwithin{equation}{section}
    \numberwithin{figure}{section}
    
    \def\cW{\mathcal{W}}
    \def\cS{\mathcal{S}}
    
    \def\cT{\mathcal{T}}
    \def\cC{\mathcal{C}}

    \def\cR{\mathcal{R}}
    
    \def\R{\mathbb{R}}

    \def\Z{\mathbb{Z}}
    \def\N{\mathbb{N}}

    \renewcommand\leq{\leqslant}
    \renewcommand\geq{\geqslant}

    \newcommand{\Tile}{\operatorname{Tile}}

    \theoremstyle{plain}
    \newtheorem{thm}{Theorem}[section]
    \newtheorem{theorem}[thm]{Theorem}
    
    \newtheorem{lemma}[thm]{Lemma}
    \newtheorem{corollary}[thm]{Corollary}
    
    \newtheorem{proposition}[thm]{Proposition}
    
    \newtheorem{question}[thm]{Question}

    \newtheorem*{claim*}{Claim}

    \theoremstyle{definition}
    \newtheorem{definition}[thm]{Definition}
    \newtheorem*{definition*}{Definition}
    \newtheorem*{remarks*}{Remarks}
    \newtheorem*{remark*}{Remark}
    \newtheorem{remark}[thm]{Remark}
    \newtheorem{example}[thm]{Example}

\newenvironment{enumerate-math}
{\begin{enumerate}
		\addtolength{\itemsep}{5pt}
		}
	{\end{enumerate}}

\begin{document}

	\title{Undecidability of translational monotilings}

	\author{Rachel Greenfeld}
	\address{School of Mathematics, Institute for Advanced Study, Princeton, NJ 08540.}
	\email{greenfeld.math@gmail.com}
	\author{Terence Tao}
	\address{UCLA Department of Mathematics, Los Angeles, CA 90095-1555.}
	\email{tao@math.ucla.edu}

	\subjclass[52C23, 03B25]{52C23, 03B25}
	\date{}
	
	
	\begin{abstract} 
      In the 60's,  Berger famously showed that translational tilings of $\Z^2$ with multiple tiles  are algorithmically undecidable. Recently, Bhattacharya proved the decidability of \emph{translational monotilings}  (tilings by translations of a single tile) in $\Z^2$. 
      The decidability of translational monotilings in higher dimensions remained unsolved. 
      In this paper, by combining  our recently developed techniques with ideas  introduced by Aanderaa--Lewis, we finally settle this problem, achieving the undecidability of translational monotilings of (periodic subsets of)  virtually $\Z^2$ spaces, namely, spaces of the form $\Z^2\times G_0$, where $G_0$ is a finite Abelian group. This also implies the undecidability of translational monotilings in $\Z^d$, $d\geq 3$.
	\end{abstract}
	
	\maketitle

\tableofcontents
	
\section{Introduction}

The study of the decidability of tilings goes back to H. Wang \cite{wang}, who introduced  the \emph{Wang tiling problem} (also known as the \emph{Wang domino problem}). 
Given a finite set of \emph{Wang tiles} $\cW$, a set of unit squares whose sides are colored, the Wang tiling problem is  the problem of determining whether it is possible to cover the plane by translated copies of the given Wang tiles without overlaps (up to null sets), under the ``domino'' constraint: two adjacent tiles in the tiling must have their colors agree on the overlapping sides. Any solution to the problem is  a \emph{Wang tiling} by the given tile-set  (see Definition  \ref{def:wang} in Section \ref{sec:2} for a formal and precise description on the Wang tiling problem).  Wang,  in his endeavours to solve the decidability of first order logic $\forall\exists\forall$ formulae, introduced a geometric interpretation of these formulae as Wang tiling problems; this motivates the question  of determining the (algorithmic) decidability of the Wang tiling problem: whether there exists an algorithm that, when given an arbitrary Wang tile-set, computes in finite time if the corresponding Wang tiling problem is solvable.

As an important milestone towards the resolution of this question, the \emph{fixed Wang tiling problem} of determining whether an arbitrary Wang tile-set $\cW$ admits a Wang tiling that contains a specified tile $w\in \cW$, was proven to be undecidable \cite{wang, buchi, kmw}.
Building on this result, a well-known work of Berger \cite{Ber, Ber-thesis} proves the undecidability of Wang tilings. In his proof, Berger shows that the Wang tiling problem is \emph{Turing complete}: he encodes any Turing machine as a Wang tiling problem, such that the Turing machine halts if and only if the Wang tiling problem is not solvable.  Then, the undecidability of the Wang tiling problem follows from the undecidability of the \emph{halting problem} \cite{turing}. See \cite{jv} and \cite[Section 1.1]{jr} for comprehensive surveys on the study of the decidability of the Wang tiling problem.

It was observed in \cite{golomb} that any Wang tiling can be encoded as a special type of \emph{translational tiling} of $\Z^2$, with the same number of tiles (see \cite[Section 1.4]{GT21} for a definition of translational tilings with multiple tiles and further discussion).  Thus, the  results of Berger on the undecidability of the Wang tiling problem  imply the undecidability of translational tilings with multiple tiles in $\Z^2$. However, 
as the Wang tiling problem with a fixed number of Wang tiles is trivially decidable\footnote{Indeed, for any fixed number $J\in \N$, there are only finitely many sets of $J$ Wang squares (up to relabeling).}, Berger's result does not imply the undecidability of translational tiling with a fixed number of tiles. 
In \cite{ollinger11}, building on Berger's argument,  Ollinger  obtained the undecidability of translational tilings with  $11$ tiles in $\Z^2$. 
Recently, in  \cite{GT21}, we proved the undecidability of translational tilings with only $2$ tiles, at the cost of enlarging the group from being two-dimensional to being \emph{virtually} two-dimensional, namely, of the form $\Z^2\times G_0$ for finite Abelian group $G_0$.
Our proof in \cite{GT21} consists of encoding any Wang tiling problem, with arbitrary number of tiles, as a translational tiling problem with merely two tiles in $\Z^2\times G_0$ (for an appropriate finite Abelian group $G_0$, that depends of the given Wang tiling problem), such that the Wang tiling problem is solvable if and only if the two tiles admit a translational tiling of $\Z^2\times G_0$.\footnote{In \cite{GT21} we actually obtained the undecidability of translational tilings of  $\Z^2\times E$, for $E\subset G_0$, with two tiles; however, using our more advanced encoding in \cite{GT22}, this result can be lifted to give the undecidability of translational tilings of the whole group $\Z^2\times G_0$ with two tiles.}  Then, the undecidability of the Wang tiling problem implies the existence of an undecidable translational tiling problem with two tiles in virtually two dimensional Abelian groups. The decidability of \emph{translational monotilings} - translational tilings with a single tile - remained unsolved. In this paper we finally settle this problem.

\subsection{Translational monotilings}\label{monotile-sec}
Let $G = (G,+)$ be a finitely generated Abelian group, $F$ be a finite subset of $G$ and $E$ be a \emph{strongly periodic} subset of $G$ (i.e., $E$ is invariant under translations by a subgroup of $G$ of finite index). We say that $F$ tiles $E$ by translations along $A\subset G$ and write
$$A\oplus F=E,$$
if the translations of $F$ along $A$:
$$a+f,\quad a\in A,\; f\in F,$$
cover each point of $E$ exactly once.  In this case we refer to  $A$ as a \emph{translational monotiling} of $E$ by the \emph{translational  monotile} $F$. 

\begin{example}\label{simple-tile}  Let $G\coloneqq \Z^2$, $F \coloneqq \{0\}\times \{0,1\}$, and $E \coloneqq \Z^2$.  Then  $A$ is a translational monotiling of $E$ by $F$ if and only if it is of the form
\begin{equation}\label{A-def}
A \coloneqq \{ (n, a(n) + m): n\in \Z ,m \in 2\Z \}
\end{equation}
with  $a \colon \Z \to \{0,1\}$. 
\end{example}

\subsection{Undecidability and aperiodicity}

The decidability of monotilings is closely related to the \emph{periodic tiling conjecture} \cite{stein, grunbaum-shephard, LW} which asserts that translational monotilings cannot be aperiodic.  A monotiling $A$ is said to be \emph{aperiodic}\footnote{In the aperiodic order literature, this notion is referred to as \emph{weak aperiodicity}.} if it cannot be replaced with (or ``repaired to'') a monotiling $A'$ which is strongly periodic (i.e., invariant with respect to a finite index subgroup).

\begin{example}  We continue the discussion of Example \ref{simple-tile}. The monotilings \eqref{A-def} are all invariant under translations by the subgroup $\{0\}\times 2\Z$ of $G$; however, this is not a finite index subgroup of $G$, and most of these monotilings\footnote{These tilings are sometimes called \emph{weakly periodic}, \emph{$1$-periodic} or \emph{half-periodic tilings} in the literature.} are not strongly periodic.  However, any such monotiling $A$ can be ``repaired'' by replacing it with the strongly periodic monotiling $\Z \times 2\Z$ (one can think of this operation as an application of appropriate vertical shifts to $A$).  Hence these tilings are not aperiodic. 
\end{example}

H. Wang \cite{wang} famously showed that if the periodic tiling conjecture holds in a class of finitely generated Abelian groups $G$ then there is an algorithm that, when given any finite subset $F$ of such a group $G$ and a strongly periodic $E\subset G$, computes in finite time whether  $F$ tiles $E$ by translations; see \cite[Appendix A]{GT21}. Thus, any undecidability result immediately implies the existence of an aperiodic tiling. In the converse direction, while aperiodicity does not imply undecidability, an undecidability result is usually foreshadowed by a construction of an aperiodic tiling. 

The periodic tiling conjecture was verified for $\Z^2$ in \cite{BH} (see also \cite{GT20} for an alternative proof); from the arguments of Wang, this implies that the translational monotiling problem in $\Z^2$ is decidable.  However, the situation changes in larger groups, such as $\Z^d$ for $d>2$, or $\Z^2 \times G_0$ for a finite Abelian group $G_0$; in our recent work \cite{GT22} we constructed aperiodic translational monotilings, showing that the periodic tiling conjecture could fail in both of these cases.
Our construction in \cite{GT22} consisted of the following two major steps:
\begin{itemize}
    \item For a sufficiently large power $q=2^s$ of $2$, we introduced the notion of a \emph{$q$-adic Sudoku puzzle} \cite[Definition 7.4]{GT22}, and showed that this puzzle did not admit any (strongly) periodic solutions; hence, this is an aperiodic puzzle \cite[Theorem 7.7]{GT22}.
    \item We developed a \emph{tiling language} and showed that the $q$-adic Sudoku puzzle is \emph{weakly expressible} (in the sense of \cite[Definition 4.13]{GT22}) in this language \cite[Theorem 4.14]{GT22}.  Thus, the $q$-adic Sudoku puzzle could be encoded as a monotiling problem.
\end{itemize}
The $q$-adic nature of the Sudoku puzzle forces a particular hierarchical structure: any solution to the puzzle has infinitely many $q$-adic scales, and its ``post-Tetris move''  (obtained by erasing its lowest scale cells) is again a solution to the puzzle \cite[Proposition 9.7]{GT22}.

Hierarchical structure is well-known and extensively used in the aperiodic order literature and in study of aperiodic tilings; for instance, any substitution tiling is a hierarchical tiling (see, e.g., \cite{ grunbaum-shephard, SM, GS1} and the references therein).  In fact, it is this hierarchical structure that plays a major role in many of the proofs of the undecidability of the Wang tiling problem and tilings with multiple tiles \cite{Ber, Ber-thesis, Rob, al, Lew, SM, Oll, drs}. In particular, a use of $p_1\times p_2$-adic structure (with sufficiently large distinct primes $p_1,p_2$), which is similar in nature to the structure of the $q$-adic-Sudoku puzzle solutions  in \cite{GT22}, was employed by Aanderaa and Lewis \cite{al, Lew}  to establish the undecidability of an empty distance subshift problem, which in turn gave an alternative proof of the undecidability  of the Wang tiling problem; see \cite[Section 4]{jv} for more details and further discussion. A key feature of the Aanderaa--Lewis construction was a decoration of each $p_1\times p_2$-scale of the subshift, which simulates the Wang tile covering the point in $\Z^2$ that is represented by this $p_1\times p_2$-scale.

\subsection{Results}

Inspired by \cite{al, Lew}, in this paper we introduce the notion of \emph{decorated $p_1\times p_2$-Sudoku puzzles} and show that any given Wang tiling problem\footnote{In fact, we can encode a broader class of domino problems, as described in Section \ref{sec:2}.} can be encoded as a Sudoku puzzle of this type, in the sense that the given Wang tiling problem is solvable if and only if the corresponding decorated $p_1\times p_2$-Sudoku puzzle is solvable. Then, using the tiling language we developed in \cite{GT21, GT22}, we encode any such Sudoku puzzle as a translational monotiling problem in a virtually $\Z^2$ space (i.e., a space of the form $\Z^2\times G_0$, where $G_0$ is a finite Abelian group). This gives the following undecidability result for translational monotilings:

\begin{theorem}[Undecidability of translational monotilings  in $\Z^2\times G_0$]\label{thm:main}
    There is no algorithm which upon any input of 
    \begin{itemize}
        \item a finite Abelian group $G_0$;
        \item a set $E\subset G_0$;
        \item a finite set $F\subset \Z^2\times G_0$,
    \end{itemize}
    computes (in finite time) whether $F$ tiles $\Z^2\times E$.
\end{theorem}

This solves \cite[Question 10.5]{GT22} in the affirmative.

\begin{remark}
    In fact, we obtain the slightly stronger result that translational monotilings in virtually $\Z^2$ spaces are \emph{Turing complete} (with \emph{Turing degree} $0'$). Indeed, in \cite{Ber, Ber-thesis} it was shown that any Turing machine can be simulated as a Wang tiling problem, giving that the Wang tiling problem is Turing complete. In this paper we eventually ``program'' the Wang tiling problem as a translational monotiling problem; thus, we obtain the Turing completeness of the latter translational monotiling problem from Turing completeness of  the Wang tiling problem.
\end{remark}

By ``pulling back'' the undecidable problem of Theorem \ref{thm:main} (using \cite[Theorem 1.1]{bgu} or the argument in \cite[Section 9]{GT21}) we obtain:

\begin{corollary}[Undecidability of translational monotilings  in $\Z^d$]\label{cor:mainZd}
    There is no algorithm which upon any input of 
    \begin{itemize}
        \item $d\geq 3$;
        \item a periodic set $E\subset \Z^d$;
        \item a finite set $F\subset \Z^d$,
    \end{itemize}
    computes (in finite time) whether $F$ tiles $E$.
\end{corollary}

\begin{remark}
In terms of \emph{logical undecidability} (or independence of ZFC), Theorem \ref{thm:main} implies the existence of an explicit finite Abelian group $G_0$, a subset $E\subset G_0$ and a finite $F\subset \Z^2\times G_0$ such that the tiling equation $\Tile(F;\Z^2\times E)$ is logically undecidable. Similarly, Corollary \ref{cor:mainZd} implies the existence of an explicit $d\geq 3$, a periodic subset $E$ of $\Z^d$ and a finite $F\subset \Z^d$ such that the tiling equation $\Tile(F;\Z^d)$ is logically undecidable. See \cite[Section 1.1]{GT21} for an short introduction to the notion of logical decidability and its correlation with the notion of algorithmic decidability.
\end{remark}

By applying \cite[Lemma 2.2]{GT22}, we can encode tiling problems in $\Z^d$ as tiling problems in $\R^d$, where the tile is now a finite union of lattice cubes, and obtain an analogue of Corollary \ref{cor:mainZd} in $\R^d$ for such monotilings (after making suitable modifications to the definitions as we are now in a continuous setting rather than a discrete one).  
Furthermore, one can apply the \emph{folded bridge construction} introduced in \cite[Section 2]{GK} to encode translational monotiling problems in $\Z^d$ as a translational monotiling problems with  \emph{connected}\footnote{See \cite[Definition 2.1]{GK}.} monotiles in $\Z^{d+2}$; by another application of \cite[Lemma 2.2]{GT22}, this would yield the undecidability of the translational monotiling problem  for connected monotiles in $\R^d$. 
We leave the precise details of this analogue to the interested reader.

\subsection{The scope of the argument}
Our argument is based on a variant of the Sudoku puzzle construction used in our previous paper \cite{GT22} to produce an aperiodic translational monotiling. However, as opposed to the previous paper where our goal was to encode the aperiodic nature of the $q$-adic Sudoku puzzle as a translational monotiling problem (which then must be aperiodic as well),
here we take further advantage of a $p_1\times p_2$-adic structure and its infinitely many scales to also encode a certain \emph{domino problem} as a \emph{decorated $p_1\times p_2$-adic Sudoku problem}. This domino problem can be viewed as a  generalization of the Wang tiling problem and is, therefore, undecidable. This, in turn, gives the undecidability of  decorated $p_1\times p_2$-adic Sudoku puzzles. As in \cite{GT22}, we can now  ``program'' any such Sudoku puzzle as a translational monotiling problem, to conclude the undecidability of translational monotilings.

The paper is organized as follows:

\begin{itemize}
    \item In Section \ref{sec:2} we introduce a \emph{domino problem} and show that this problem is undecidable. 
    \item Sections \ref{sec:3}, \ref{sec:4} and \ref{sec:5} are devoted to the construction of a Sudoku-type puzzle that encodes the domino problem: 
    \begin{itemize}
        \item[] In Section \ref{sec:3} we introduce the notion of \emph{$(\cS,\cC)$-Sudoku puzzles}, which generalizes the  construction in \cite[Section 7]{GT22}.
        \item[] In Section \ref{sec:4}, as an instance of a $(\cS,\cC)$-Sudoku puzzle introduced in the previous section, we  construct a \emph{$p$-adic Sudoku puzzle} (for a sufficiently large prime $p$) and explore its properties.
        \item[] Building on that, in Section \ref{sec:5} we finally construct a \emph{decorated $p_1\times p_2$-Sudoku puzzle} (which is another special case of a $(\cS,\cC)$-Sudoku puzzle), and show that any domino problem can be encoded as such a Sudoku puzzle.
    \end{itemize}    
    \item Finally, in Section \ref{sec:6}, using our tiling encoding approach we previously developed in \cite{GT21, GT22}, we show that any  $(\cS,\cC)$-Sudoku puzzle can be encoded as a translational monotiling problem in a virtually $\Z^2$ space.
\end{itemize}

A high-level roadmap for the argument in the paper is given in Figure \ref{fig:logic1}.

\subsection{Notation} 

The natural numbers $\N = \{0,1,2,\dots\}$ will start at $0$ in this paper.  For any modulus $q \geq 1$, we let $\pi_q \colon \Z \to \Z/q\Z$ be the projection homomorphism $\pi_q(n) \coloneqq n \pmod{q}$.

If $A \subset G$ is a subset of an Abelian group $G = (G,+)$ and $h \in G$, we write 
$$h+A \coloneqq A+h \coloneqq \{ a+h: a \in A \}$$
for the translate of $A$, and
$$ -A \coloneqq \{-a: a \in A \}$$
for the reflection of $A$.  We write $A \uplus B$ for the disjoint union of $A$ and $B$ (defined to equal $A \cup B$ when $A,B$ are disjoint, and undefined otherwise).

\subsection{Acknowledgements} We are grateful to Emmanuel Jeandel for bringing to our attention the important references \cite{al, Lew}. We also gratefully acknowledge the hospitality and support of the Institute for Advanced Study during
the Special Year on Dynamics, Additive Number Theory and Algebraic Geometry, where a significant portion of this research was conducted. 

The first author is supported by the Association of Members of the Institute for Advanced Study (AMIAS) and by NSF grant  DMS-2242871. The second author is supported by NSF grant DMS-1764034. 

\section{Domino functions}\label{sec:2}

In this section we develop the concept of a \emph{domino function}, which can be thought of as a generalization of the concept of a Wang tiling.  In particular, the undecidability of the Wang tiling problem will imply the undecidability of the domino function problem.

Domino functions will be defined on subsets of a ``domino board'' $\Z^2$.  We set out some basic notation for this board:

\begin{definition}[The domino board]  We define the \emph{domino board} to be the standard lattice
$$ \Z^2 \coloneqq \{ (s_1,s_2): s_1, s_2 \in \Z \}$$
with generators
$$ e_1 \coloneqq (1,0); \quad e_2 \coloneqq (0,1).$$
In order to maintain the analogy with actual domino boards, we will identify $(s_1,s_2)$ in $\Z^2$ with the unit square $[s_1,s_1+1] \times [s_2,s_2+1]$ in $\R^2$, and thus refer to elements of this board as ``unit squares''.  We place the product order $\leq$ on $\Z^2$, thus $(s_1, s_2) \leq (t_1,t_2)$ if and only if $s_1 \leq t_1$ and $s_2 \leq t_2$.  If $r, t \in \Z^2$ are such that $r \leq t$, we define the \emph{rectangle} $[r,t]$ to be the set
$$ [r,t] \coloneqq \{ s \in \Z^2: r \leq s \leq t \};$$
thus, in coordinates
$$ [(r_1,r_2), (t_1,t_2)] = \{ (s_1,s_2) \in \Z^2: r_1 \leq s_1 \leq t_1; r_2 \leq s_2 \leq t_2 \}.$$
A rectangle of the form $[r,r+e_1]$ will be called a \emph{horizontal domino tile}, and a rectangle of the form $[r,r+e_2]$ will be called a \emph{vertical domino tile}, with these two types of rectangles collectively referred to as \emph{domino tiles}; thus a domino tile consists of two adjacent unit squares.
\end{definition}

One can think of the domino tiles as describing the Cayley graph structure on $\Z^2$ induced by the generators $e_1, e_2$.

\begin{definition}[Domino function]\label{def:dom}  A \emph{domino set}
$$ \cR = (\cW, \cR_1, \cR_2)$$
is an ordered triple consisting of a non-empty finite set $\cW$ (whose elements we call \emph{pips}) and a pair of subsets $\cR_1, \cR_2$ of $\cW^2$, which we call the \emph{horizontal domino set} and \emph{vertical domino set} respectively. If $\cR = (\cW, \cR_1, \cR_2)$ is a domino set and $\Omega \subset \Z^2$ is a portion of the domino board, then a \emph{$\cR$-domino function} on $\Omega$ is a function
    $$\cT \colon \Omega \to \cW$$
    that places a pip $\cT(s) \in \cW$ on every unit square $s$ in $\Omega$, with the property
    $$ (\cT(s), \cT(s+e_i)) \in \cR_i$$
    whenever $[s,s+e_i]$ is a domino tile in $\Omega$.  We say that the \emph{$\cR$-domino problem is solvable on $\Omega$} if there exists at least one $\cR$-domino function on $\Omega$.
\end{definition}

Informally, the horizontal and vertical domino sets $\cR_1, \cR_2$ describe the permitted values of the domino function on horizontal and vertical domino tiles respectively.  A $\cR$-domino tiling on a set $\Omega$ is then an assignment of pips in $\cW$ to each square in $\Omega$ in such a way that every domino tile in this collection is permitted.

\begin{example}[Semistandard Young tableaux]  Let $\Omega \subset \N^2 \subset \Z^2$ be a Young shape, which we orient in the ``French'' style with the longest row at the bottom (see Figure \ref{fig:semistandard}).  A semi-standard Young tableau on $\Omega$ using the numbers $\{1,\dots,k\}$ (i.e., a labeling of $\Omega$ by numbers in $\{1,\dots,k\}$ that is weakly increasing along each row and strongly increasing up each column), is precisely an $\cR$-domino function on $\Omega$, where the domino set $\cR = (\cW,\cR_1,\cR_2)$ is defined by selecting the set of pips $\cW = \{1,\dots,k\}$, the horizontal domino set
$$ \cR_1 \coloneqq \{ (i,j) \in \cW^2: i \leq j \}$$
and the vertical domino set
$$ \cR_2 \coloneqq \{ (i,j) \in \cW^2: i < j \}.$$
The $\cR$-domino problem on $\Omega$ is solvable if and only if the Young shape $\Omega$ has at most $k$ rows.
\end{example}

\begin{figure}
    \centering
\centerline{\includegraphics[height=3in]{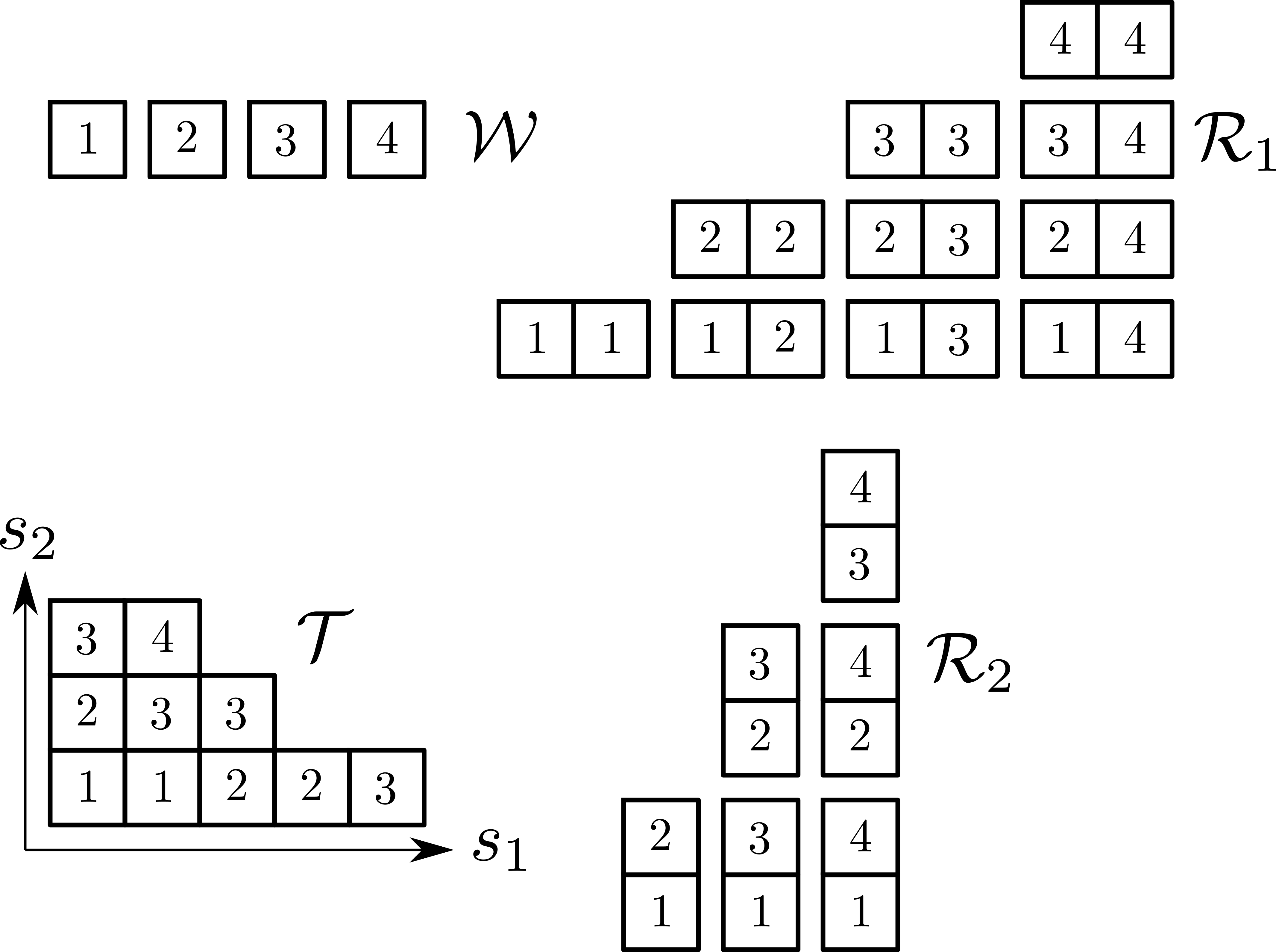}}
    \caption{A semistandard Young tableau $\cT$ on the Young shape $\Omega$ associated to a partition $(5,3,2)$, with $k=4$ (so $\cW = \{1,2,3,4\}$ is the set of pips), together with the horizontal domino set $\cR_1$ and the vertical domino set $\cR_2$, where elements of $\cW^2$ are identified with labelings of horizontal or vertical domino tiles in the obvious fashion.  Note how every horizontal or vertical domino tile in $\Omega$ is labeled by $\cT$ with an element of the domino set associated with the horizontal or vertical domino set respectively.}
    \label{fig:semistandard}
\end{figure}

\begin{remark}[Sheaf-theoretic interpretation]  One can view $\cR$-domino functions as sections of a certain sheaf, defined as follows.  Let 
$$ X \coloneqq \Z^2 \uplus \{ [s,s+e_l]: s \in \Z^2; l=1,2 \}$$
be the disjoint union of the domino board $\Z^2$ with the collection of its domino tiles.
We impose a (non-Hausdorff) topology on $X$ by defining an open set to be a set of the form
$$ U_\Omega \coloneqq \Omega \uplus \{ [s,s+e_l]: l=1,2; s, s+e_l\in \Omega \}$$
for some $\Omega \subset \Z^2$ (thus $U_\Omega$ consists of the set $\Omega$ and its domino tiles); one easily verifies that this collection obeys the axioms of a topology.  Every $\cR$-domino function $\cT \colon \Omega \to \{1,\dots,k\}$ on a subset $\Omega$ of the domino board $\Z^2$ defines a function
$$ \overline\cT \colon U_\Omega \to \{1,\dots,k\} \cup \cR_1 \cup \cR_2$$
by the formulae
\begin{align*}
\overline\cT(s) &\coloneqq \cT(s)\\    
\overline\cT([s,s+e_l]) &\coloneqq (\cT(s), \cT(s+e_l))
\end{align*}
for $s \in \Omega$ and $l=1,2$; we then define $F(U_\Omega)$ to be the collection of all functions $\overline\cT$ obtained in this fashion.  One can then verify that these collections $F(U_\Omega)$ obey the axioms of a sheaf over $X$ (with the usual restriction operators, and with stalks identified with the set of pips $\cW$), thus a $\cR$-domino function on $\Omega$ is identified with a section of the sheaf over $U_\Omega$.    
However, we will not use the language of sheaves further here.  We remark that a related sheaf structure on more general graphs than $\Z^2$ (in which the stalks are now vector spaces, and the domino sets are replaced by linear compatibility conditions) appears in \cite{Fried}.
\end{remark}

\subsection{The domino problem}

We have the following simple equivalence:

\begin{lemma}[Equivalent forms of the domino problem]\label{dom-equiv}  Let $\cR$ be a domino set.  Then the following are equivalent:
\begin{itemize}
    \item[(i)] The domino problem is solvable on $\Z^2$.
    \item[(ii)] The domino problem is solvable on $\N^2$.
    \item[(iii)]  The domino problem is solvable on the rectangle $[(0,0),r]$ for any $r \geq (0,0)$.
    \item[(iv)]  The domino problem is solvable on the rectangle $[-r,r]$ for any $r \geq (0,0)$.
\end{itemize}
\end{lemma}

\begin{proof} From the inclusions $[(0,0),r]\subset \N^2\subset \Z^2$ we see that (i) implies (ii) implies (iii).  The implication of (iv) from (iii) follows by replacing $r$ by $2r$ and then translating by $-r$.  Finally, the implication of (i) from (iv) follows from the compactness theorem in logic (or from other standard compactness theorems, such as Tychonoff's theorem or the Arzel\'a--Ascoli theorem); a closely related way to proceed is to perform a limit along an ultrafilter.
\end{proof}

 $\cR$-domino problems include \emph{Wang tiling problems} as a special case.  We pause to set up the notation for Wang tilings.

\begin{definition}[Wang tilings]\label{def:wang}  Let $H, V$ be finite sets (of ``colors'').  A \emph{Wang tile} is a quadruplet $(t_W, t_E, t_S, t_N)$ in $H^2 \times V^2$.  Given a collection $\cW \subset H^2 \times V^2$ of Wang tiles, a \emph{$\cW$-Wang tiling} is an assignment $\cT \colon \Z^2 \to \cW$ of a Wang tile
$$ \cT(s) = (\cT_W(s), \cT_E(s), \cT_S(s), \cT_N(s)) \in \cW$$
in $\cW$ to each unit square $s$ in the domino board $\Z^2$, such that one has the compatibility conditions
$$    
\cT_E(s) = \cT_W(s+e_1); \quad \cT_N(s) = \cT_S(s+e_2)
$$
for all $s \in \Z^2$.  We say that the $\cW$-Wang tiling problem is \emph{solvable} if there exists at least one $\cW$-Wang tiling.  
\end{definition}

Suppose one has a collection $\cW \subset H^2 \times V^2$ of Wang tiles
$$ t = (t_{W}, t_{E}, t_{S}, t_{N}).$$
One can then associate a pair $(\cR_1, \cR_2)$ of subsets $\cR_1, \cR_2$ of $\cW^2$ by the formulae
\begin{align}
\cR_1 &\coloneqq \{ (t,t') \in \cW^2: t_{E} = t'_{W} \}\label{r1-def} \\
\cR_2 &\coloneqq \{ (t,t') \in \cW^2: t_{N} = t'_{S} \}.\label{r2-def}
\end{align}

By inspecting Definitions \ref{def:dom} and \ref{def:wang}, we see that a function $\cT \colon \Z^2 \to \cW$ is a $\cW$-Wang tiling if and only if it is a $(\cW, \cR_1, \cR_2)$-domino function.  In particular, the $\cW$-Wang tiling problem is solvable if and only if the $(\cW, \cR_1, \cR_2)$-domino problem is solvable on $\Z^2$.  Employing the celebrared result of Berger \cite{Ber, Ber-thesis, Rob, drs} that there is no algorithm that, when given an arbitrary Wang tile-set $\cW$, decides in finite time whether the $\cW$-Wang tiling problem solvable (see also \cite{K07, al, Lew, SM} for alternative proofs).  Combining this with the above correspondence and Lemma \ref{dom-equiv}, we conclude

\begin{corollary}[Undecidability of the domino problem]\label{undecided}  There is no algorithm that, when given an arbitrary domino set $\cR$, decides in finite time whether the $\cR$-domino tiling problem is solvable on $\N^2$.
\end{corollary}

\begin{remark}  While all Wang tiling problems can be encoded as domino problems, the converse is not clear.  The horizontal and vertical domino sets $\cR_l, l=1,2$ formed from the construction \eqref{r1-def}, \eqref{r2-def} are not arbitrary subsets of $\cW^2$, but obey the following ``rectangular closure property'': if $\cR_l$ contains three vertices $(a,b), (a,b'), (a',b)$ of a ``rectangle'' in $\cW$, then it must also contain the fourth vertex $(a',b')$.  Indeed, if  $(a,b), (a,b'), (a',b)$ are in $\cR_1$ then 
$$a'_E = b_W = a_E = b'_W,$$ 
hence $(a',b')\in\cR_1$.  Similarly for $\cR_2$.
It is easy to construct domino problems which do not obey this rectangular closure property, and which are thus not obviously encodable as a Wang tiling problem.  We will however not use such ``exotic'' domino problems in this paper, as we will only need the domino problem provided by Corollary \ref{undecided} which does come from a Wang tiling and thus obeys the rectangular closure property.
\end{remark}

\section{Sudoku puzzles}\label{sec:3}

We now recall the notion of a \emph{$\cS$-Sudoku puzzle} from \cite{GT22}.

\begin{definition}[Sudoku puzzles]\label{sudoku-def}  Let $N$ be a natural number (which we call the \emph{width}), and let $\Sigma$ be a non-empty finite set (the \emph{digit set}).  The \emph{Sudoku board} $\mathbb{B} = \mathbb{B}_N$ is the set
$$ \mathbb{B} \coloneqq \{0,\dots,N-1\} \times \Z.$$
Elements $(n,m)$ of the Sudoku board $\mathbb{B}$ will be referred to as \emph{cells}; conceptually we keep them separate from the unit squares $(s_1,s_2)$ of the domino board $\Z^2$, as they will play a rather different role.  We isolate some collections of cells of relevance to our arguments:
\begin{itemize}
    \item A \emph{column} is a set of cells of the form $\{n\} \times \Z$ for some $0 \leq n \leq N-1$.
    \item A \emph{non-vertical line} $\ell = \ell_{j,i}$ is a set of cells of the form
$$ \ell_{j,i} \coloneqq \{ (n,jn+i): 0 \leq n \leq N-1 \}$$
    for some \emph{slope} $j \in \Z$ and \emph{intercept} $i \in \Z$.
    \item A \emph{row} is a non-vertical line of slope $0$, that is to say a set of cells of the form $\{0,\dots,N-1\} \times \{m\}$ for some $m \in \Z$.
    \item A \emph{diagonal} is a non-vertical line of slope $1$, that is to say a set of cells of the form $\{ (n, n+i): 0 \leq n \leq N-1\}$ for some $i \in \Z$.
    \item An \emph{anti-diagonal} is a non-vertical line of slope $-1$, that is to say a set of cells of the form $\{ (n, i-n): 0 \leq n \leq N-1 \}$ for some $i \in \Z$.
    \item A \emph{square} $Q_{n_0,m_0}$ is a set of cells of the form\footnote{In \cite{GT22} the squares had sidelength $8$, but for our current arguments it is more natural to take squares of sidelength $4$.} 
    \begin{equation}\label{square}
    Q_{n_0,m_0} \coloneqq \{n_0,\dots,n_0+3\} \times \{m_0,\dots,m_0+3\}
    \end{equation}
    for some $0 \leq n_0 \leq N-4$ and $m_0 \in \Z$.
\end{itemize}

See Figure \ref{fig:board_p}.

\begin{figure}
    \centering
    \includegraphics[width = .9\textwidth]{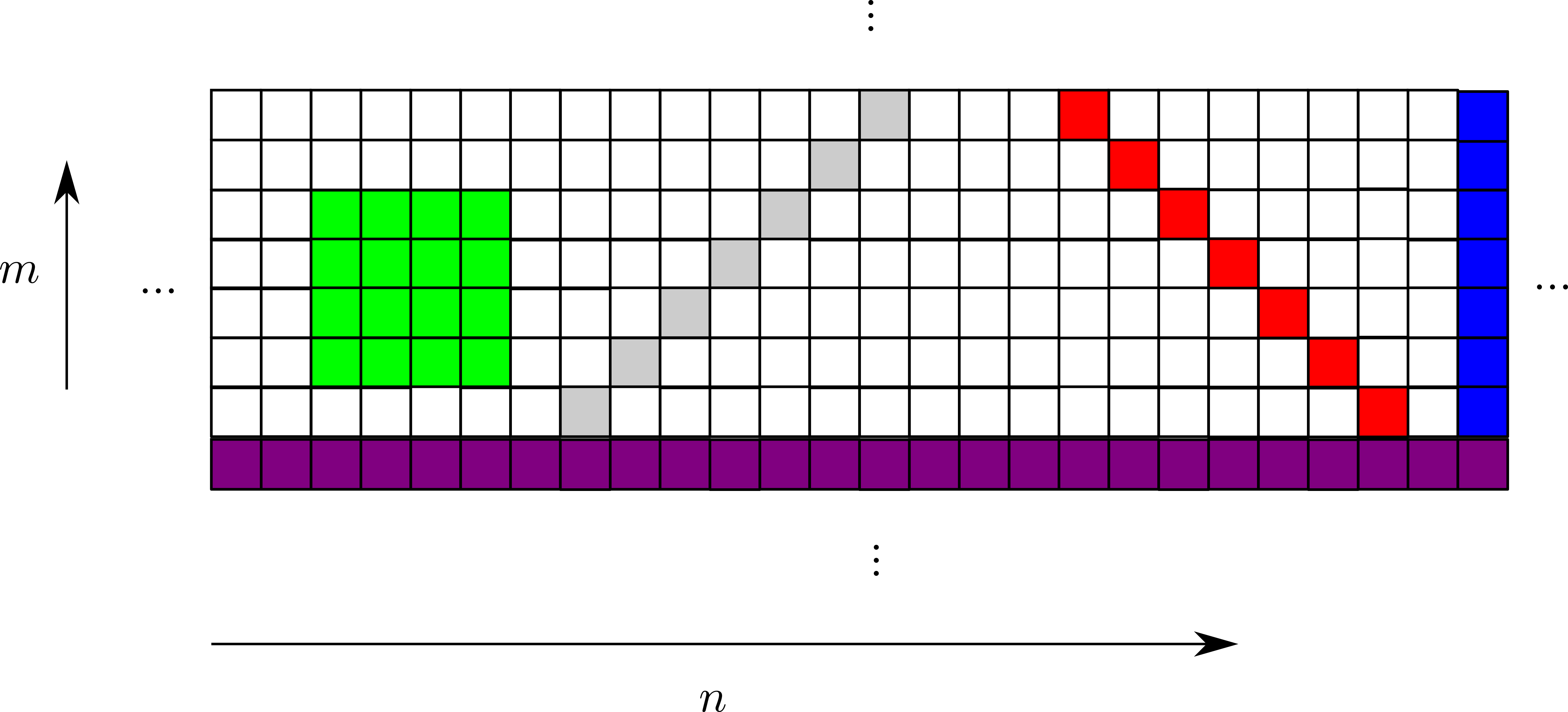}
    \caption{A portion of the Sudoku board $\mathbb{B}$, with some selected (overlapping) objects: a column (in blue), a row (in purple), a diagonal (in gray), an antidiagonal (in red), and a square (in green). }
    \label{fig:board_p}
\end{figure}

A \emph{Sudoku rule} $\cS= \cS_{N,\Sigma}$ with width $N$ and digit set $\Sigma$, is a collection of functions $$g \colon \{0,\dots,N-1\} \to \Sigma;$$ thus $\cS$ can be identified with a subset of $\Sigma^N$.  If $\cS$ is a Sudoku rule, a \emph{$\cS$-Sudoku solution} is a function $F \colon \mathbb{B} \to \Sigma$ with the property that for each non-vertical line $\ell_{j,i}$, the function
$$ n \mapsto F(n, jn+i),$$
(which parameterizes the restriction of $F$ to the line $\ell_{j,i}$) lies in $\cS$.

An \emph{initial condition} $\cC = \cC_q$ with period $q \geq 1$, is a set $\cC \subset (\Z/q\Z) \times \Sigma$. A $\cS$-Sudoku solution $F \colon \mathbb{B} \to \Sigma$ is said to obey the \emph{initial condition $\cC$}  if, for every column $\{n\} \times \Z$ in $\mathbb{B}$, there exists a permutation $\sigma_n \colon \Z/q\Z \to \Z/q\Z$ such that
$$ ( \sigma_n(\pi_q(m)), F(n,m) ) \in \cC$$
for all $(n,m)$ in the column.  We say that the $(\cS, \cC)$-Sudoku puzzle is \emph{solvable} if there exists a $\cS$-Sudoku solution $F \colon \mathbb{B} \to \Sigma$ which obeys the initial condition $\cC$. 
\end{definition}

As in \cite{GT22}, one can think of a Sudoku rule $\cS$ as analogous to the requirement in a traditional Sudoku puzzle that the digits (in the digit set $\Sigma=\{1,\dots,9\})$ assigned to every row, column, and $3 \times 3$ block in a $9 \times 9$ grid form a permutation.  The initial condition $\cC$ is somewhat analogous to the initial digits that are already filled in at the start of the puzzle.  The ``good columns'' condition appearing in \cite{GT22} can be viewed as analogous to a special case of an initial condition $\cC$.

Theorem \ref{thm:main} will now follow from combining Corollary \ref{undecided} with the following two theorems, as illustrated in  Figure \ref{fig:logic1}.

\begin{theorem}[Encoding domino problems as Sudoku puzzles]\label{domino-to-sudoku} Given a domino set $\cR$, one can generate (in finite time) a Sudoku rule $\cS$ and an initial condition $\cC$, such that the $\cR$-domino problem is solvable on $\N^2$ if and only if the $(\cS, \cC)$ Sudoku puzzle is solvable.
\end{theorem}

\begin{theorem}[Encoding Sudoku puzzles as monotiling problems]\label{sudoku-to-monotile} Given a Sudoku rule $\cS$ and an initial condition $\cC$, one can generate (in finite time) a finite Abelian group $G_0$, a set $E\subset G_0$, and a finite set $F\subset \Z^2\times G_0$, with the property that the $(\cS, \cC)$ Sudoku puzzle is solvable if and only if $F$ tiles $\Z^2 \times E$.
\end{theorem}

Theorem \ref{sudoku-to-monotile} will be proven in Section \ref{sec:6} using a variant of the ``tiling language'' developed in \cite{GT22} that was implicit in \cite{GT21}.  The more novel ingredient in the proof of Theorem \ref{thm:main} is Theorem \ref{domino-to-sudoku}, which will be proven over the next two sections, as illustrated in Figures \ref{fig:sec4} and \ref{fig:sec4-5}.  For now, we close the section by observing a simple but useful affine invariance of $\cS$-Sudoku solutions (cf. \cite[Proposition 8.4(i)]{GT22}).

\begin{figure}
    \centering
    \includegraphics[width = .7\textwidth]{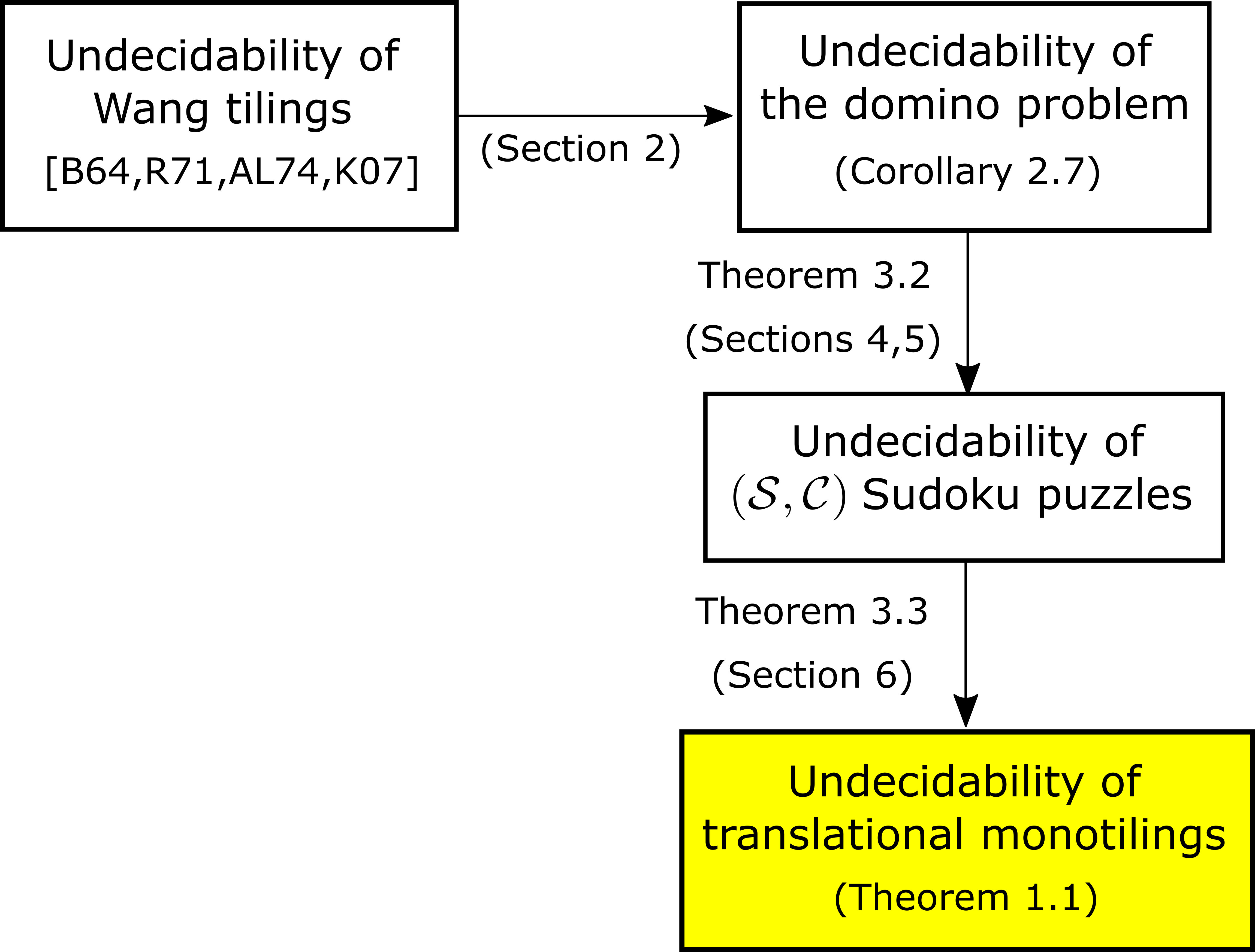}
    \caption{A high-level overview of our proof.}
    \label{fig:logic1}
\end{figure}

\begin{proposition}[Affine invariance]\label{sudoku-affine}  Let $\cS = \cS_{N,\Sigma}$ be a Sudoku rule, and let $F \colon \mathbb{B} \to \Sigma$ be a $\cS$-Sudoku solution.  Then for any integers $a,b,c$, the function $(n,m) \mapsto F(n,an+bm+c)$ is also a $\cS$-Sudoku solution.  In particular:
\begin{itemize}
    \item[(i)] The \emph{reflection} $(n,m) \mapsto F(n,-m)$ is a $\cS$-Sudoku solution.
    \item[(ii)]  For any integers $D,E$, the \emph{shear} $(n,m) \mapsto F(n, m+Dn+E)$ is a $\cS$-Sudoku solution.
    \item[(iii)]  For any prime $p$, the \emph{Tetris move outcome} $(n,m) \mapsto F(n,pm)$ is a $\cS$-Sudoku solution.
\end{itemize}
\end{proposition}

\begin{proof} Immediate from the definitions, since the affine transformation $(n,m) \mapsto (n,an+bm+c)$ maps non-vertical lines to non-vertical lines.
\end{proof}

\section{A $p$-adic Sudoku rule}\label{sec:4}

The Sudoku rule $\cS$ used to establish Theorem \ref{domino-to-sudoku} will be rather intricate, consisting of a superposition of two ``$p$-adic'' Sudoku rules which will essentially encode the two coordinates $s_1, s_2$ of the domino board, with an additional ``decoration rule'' needed to encode the domino function itself, and which is coupled with the previous two rules.  In this section we develop the theory of a single $p$-adic Sudoku.  We begin with some $p$-adic notation.

\begin{definition}[$p$-adic structures]  Let $p$ be a prime.  We let $(\Z/p\Z)^\times = (\Z/p\Z) \backslash \{0\}$ denote the invertible elements of $\Z/p\Z$, and let\footnote{In particular, we caution that $\Z_p$ does \emph{not} denote the cyclic group $\Z/p\Z$.} $\Z_p = \varprojlim \Z/p^n \Z$ denote the ring of $p$-adic numbers, with the usual projection homomorphism $\pi_p \colon \Z_p \to \Z/p\Z$.  For any $p$-adic integer $n \in \Z_p$, we define the \emph{$p$-valuation} $\nu_p(n) \in [0,+\infty]$ of $n$ to be the number of times $p$ divides $n$ if $n$ is non-zero, with the convention $\nu_p(0) \coloneqq +\infty$.  As is well known, with the $p$-adic metric $d_p( n, m) \coloneqq p^{-\nu_p(n-m)}$, $\Z_p$ becomes a compact Abelian group containing $\Z$ as a dense subgroup.  We define the ``$p$-adically structured function'' $f_p \colon \Z_p \to (\Z/p\Z)^\times$ by the formula
$$ f_p(n) \coloneqq \pi_p\left(\frac{n}{p^{\nu_p(n)}}\right)$$
for $n \neq 0$ (i.e., $f_p(n)$ is the last non-zero digit in the base $p$ expansion of $n$), with $f_p(0)$ set to an arbitrary value of $(\Z/p\Z)^\times$; for sake of concreteness, we adopt the convention
$$ f_p(0) \coloneqq 1 \pmod{p}.$$
\end{definition}

We observe the multiplicativity and additivity properties
\begin{equation}\label{fp-mult}
 f_p(nm) = f_p(n) f_p(m); \quad \nu_p(nm) = \nu_p(n) + \nu_p(m)
 \end{equation}
for any \emph{non-zero} $n,m \in \Z_p$, as well as the almost periodicity property
\begin{equation}\label{fp-period}
f_p(n+h) = f_p(n); \quad \nu_p(n+h) = \nu_p(n)
\end{equation}
whenever $n,h \in \Z_p$ are such that $\nu_p(h) > \nu_p(n)$.   In particular, $f_p$ and $\nu_p$ are continuous (in fact locally constant) on $\Z_p$ away from the origin $0$, where they both exhibit a singularity.

\begin{remark}
In \cite{GT22} we used a similar function $f_q$ to $f_p$, in which $p$ was replaced by a power of two $q=2^s$, and the domain was restricted to just the integers $\Z$ rather than the $q$-adics $\Z_q$.  
\end{remark}

It will be convenient to introduce some notation for affine forms.

\begin{definition}[Affine forms] Let  $p$ be a prime.
\begin{itemize}
    \item[(i)]  An \emph{affine form of one variable} is a function of the form $$n \mapsto an+b,$$ where $a,b$ are integers and $n$ is an indeterminate (usually restricted in practice to subset of $\Z$).  We say that the affine form is \emph{degenerate} modulo $p$ if the form $\pi_p(an+b)$ vanishes identically, or equivalently if $(\pi_p(a), \pi_p(b)) = (0,0)$, and \emph{non-degenerate} modulo $p$ otherwise.
    \item[(ii)]  An \emph{affine form of two variables} is a function of the form $$(n,m) \mapsto An+Bm+C,$$ where $A,B,C$ are integers and $n,m$ are indeterminates (but usually restricted in practice to subset of $\Z$).  We say that the affine form is \emph{degenerate} modulo $p$ if the form $\pi_p(An+Bm+C)$ vanishes identically, or equivalently if $(\pi_{p}(A), \pi_{p}(B), \pi_{p}(C)) = (0,0,0)$, and \emph{non-degenerate} modulo $p$ otherwise.  The affine form is said to be \emph{vertically non-degenerate} if $\pi_p(B) \neq 0$.
    \item[(iii)]  Two affine forms of one variable $$n \mapsto an+b,\quad n \mapsto a'n+b'$$  are said to \emph{agree modulo $p$} at a given integer $n \in \Z$ if $$\pi_{p}(an+b) = \pi_{p}(a'n+b').$$  They are said to be \emph{identical modulo $p$} if they agree at every integer $n \in \Z$, or equivalently if $a' = a \pmod{p}$ and $b' = b \pmod{p}$.
    \item[(iv)]  Two affine forms of two variables $$(n,m) \mapsto An+Bm+C, \quad (n,m) \mapsto A'n+B'm+C'$$  are said to \emph{agree modulo $p$} at a given pair $(n,m) \in \Z^2$ if $$\pi_{p}(An+Bm+C) = \pi_{p}(A'n+B'm+C').$$  They are said to be \emph{identical modulo $p$} if they agree at every element of $\Z^2$, or equivalently if 
    $$A' = A \pmod{p};\quad B' = B \pmod{p};\quad C' = C \pmod{p}.$$
\end{itemize}
\end{definition}

We record some basic facts about affine forms:

\begin{lemma}[Basic facts about affine forms]\label{basic-affine} Let $p$ be a prime. 
\begin{itemize}
    \item[(i)] (Non-degenerate forms are usually invertible, I)  Let $n \mapsto an+b$ be a non-degenerate  affine form of one variable.  Then $\pi_{p}(an+b)$ is non-zero outside of at most one coset of $p \Z$.  In particular, $\nu_{p}(an+b)=0$ outside at most one coset of $p \Z$.
    \item[(ii)]   (Non-degenerate forms are usually invertible, II) Let $(n,m) \mapsto An + Bm + C$  be a non-degenerate  affine form of two variables.  Then $\pi_{p}(An+Bm+C)$ is non-zero for $(n,m) \in \Z^2$ outside of at most $p$ cosets of $p\Z\times p\Z$. In particular, $\nu_p(An+Bm+C) = 0$ outside of at most $p$ cosets of $p\Z\times p\Z$.
    \item[(iii)]  (Extrapolation, I) Let $n \mapsto an+b$ and $n \mapsto a'n+b'$ be two  affine forms of one  variable that agree modulo $p$ at two integers $n_1, n_2$ with $\pi_p(n_1) \neq \pi_p(n_2)$.  Then the two forms are identical modulo $p$.
    \item[(iv)] (Extrapolation, II)  Suppose that $(n,m) \mapsto An+Bm+C$ and $(n,m) \mapsto A'n+B'm+C'$ are  affine forms of two variables that agree modulo $p$ on more than $p$ cosets of $p\Z\times p\Z$.  Then these affine forms are in fact identical modulo $p$.
\end{itemize}
\end{lemma}

\begin{proof}  The claim (i) is clear since $\pi_{p}(an+b)$ is an  affine form of one variable on the field $\Z/p\Z$ that does not vanish identically.  Similarly for (ii).  To prove (iv), note that if the two  affine forms were not identical modulo $p$, then their difference $(n,m) \mapsto (A'-A)n +(B'-B)m+(C'-C)$ would be an  affine form that does not vanish identically modulo $p$, and hence would vanish modulo $p$ only on at most $p$ cosets of $p\Z\times p\Z$, a contradiction.  The claim (iii) is established similarly.
\end{proof}

We can now construct the Sudoku puzzle associated to the $p$-adic structure.

\begin{definition}[Constructing a $p$-adic Sudoku puzzle]\label{sudoku-construction-p}  Let $p$ be a prime, and let $N$ be a natural number. We then construct a Sudoku rule $\cS_{p,N}$ with board width $N$ and digit set $(\Z/p\Z)^\times$, to be the collection of all functions $g \colon \{0,\dots,N-1\} \to (\Z/p\Z)^\times$ for which there exists a non-degenerate  affine form $n \mapsto an+b$ in one variable, such that
$$ g(n) =  f_p( an+b ) $$
    for all $n \in \{0,\dots,N-1\}$ with $\nu_p(an+b) \leq 1$.

A $\cS_{p,N}$-Sudoku solution $F \colon \mathbb{B} \to (\Z/p\Z)^\times$ is said to have \emph{non-constant columns} if, for every column $\{n\} \times \Z$ in $\mathbb{B}$, the function $m \mapsto F(n,m)$ is a non-constant function on $\Z$. 
\end{definition}

It turns out that when $p$ is sufficiently large, the $\cS_{p,N}$-Sudoku solutions can be almost completely classified in terms of forms $(n,m) \mapsto An+Bm+C$ that resemble affine forms of two variables, except that the coefficients $A,B,C$ take values in the $p$-adics $\Z_p$ rather than the integers $\Z$.  More precisely:

\begin{theorem}[Near-classification of $\cS_{p,N}$-Sudoku solutions]\label{class}  Let $p$ be a prime obeying the largeness condition
\begin{equation}\label{p-large}
p > 48,
\end{equation}
and let $N$ be a multiple of $p^2$.
\begin{itemize}
    \item[(i)]  If $A,B,C \in \Z_p$ are not all zero, and $F \colon \mathbb{B} \to (\Z/p\Z)^\times$ is a function such that
    \begin{equation}\label{fpc}
    F(n,m) = f_p(An+Bm+C)
    \end{equation}
    holds for all $(n,m) \in \mathbb{B}$, then $F$ is a $\cS_{p,N}$-Sudoku solution.  Furthermore, if $\nu_p(B)=0$, then $F$ has non-constant columns.
    \item[(ii)] Conversely, if  $F$ is a $\cS_{p,N}$-Sudoku solution with non-constant columns, then there exist $A,B,C \in \Z_p$ with $\nu_p(B)=0$, such that \eqref{fpc} holds for all $(n,m) \in \mathbb{B}$ with $An+Bm+C \neq 0$.
\end{itemize}
\end{theorem} 

The precise threshold of $48$ in \eqref{p-large} is what arises from our proof of Theorem \ref{class}; it can probably be lowered, but this will not significantly simplify the remainder of our arguments.  Similarly, the requirement that $N$ be a multiple of $p^2$ can be relaxed to $N \geq p^2$ without much difficulty.  It is part (ii) of the theorem which will be important in our application; we will not use part (i) other than to demonstrate that the conclusion of part (ii) is close to optimal.

\begin{figure}
    \centering
    \includegraphics[width = .5\textwidth]{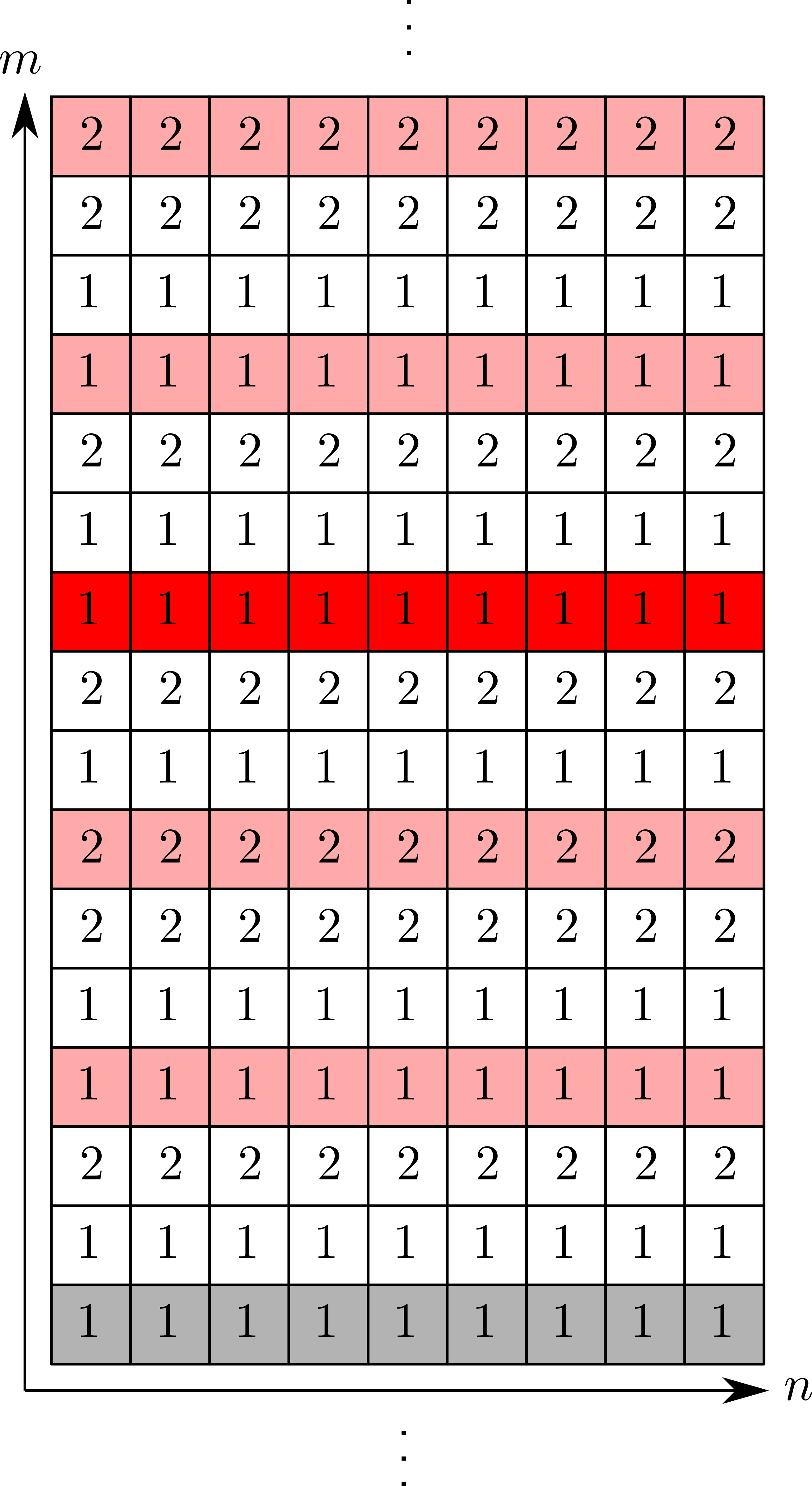}
    \caption{A horizontal slice of the $\cS_{3,9}$-Sudoku solution $F_1(n,m) \coloneqq f_3(m)$.  The colors  gray, white, pink, and red correspond to the cases $\nu_3(m)=+\infty$, $\nu_3(m)=0$, $\nu_3(m)=1$, and $\nu_3(m)=2$ respectively. One could apply affine transformations (such as shear transformations) to this solution to create further $\cS_{3,9}$-Sudoku solutions if desired.}
    \label{fig:p3}
\end{figure}

\begin{figure}
    \centering
    \includegraphics[width = .9\textwidth]{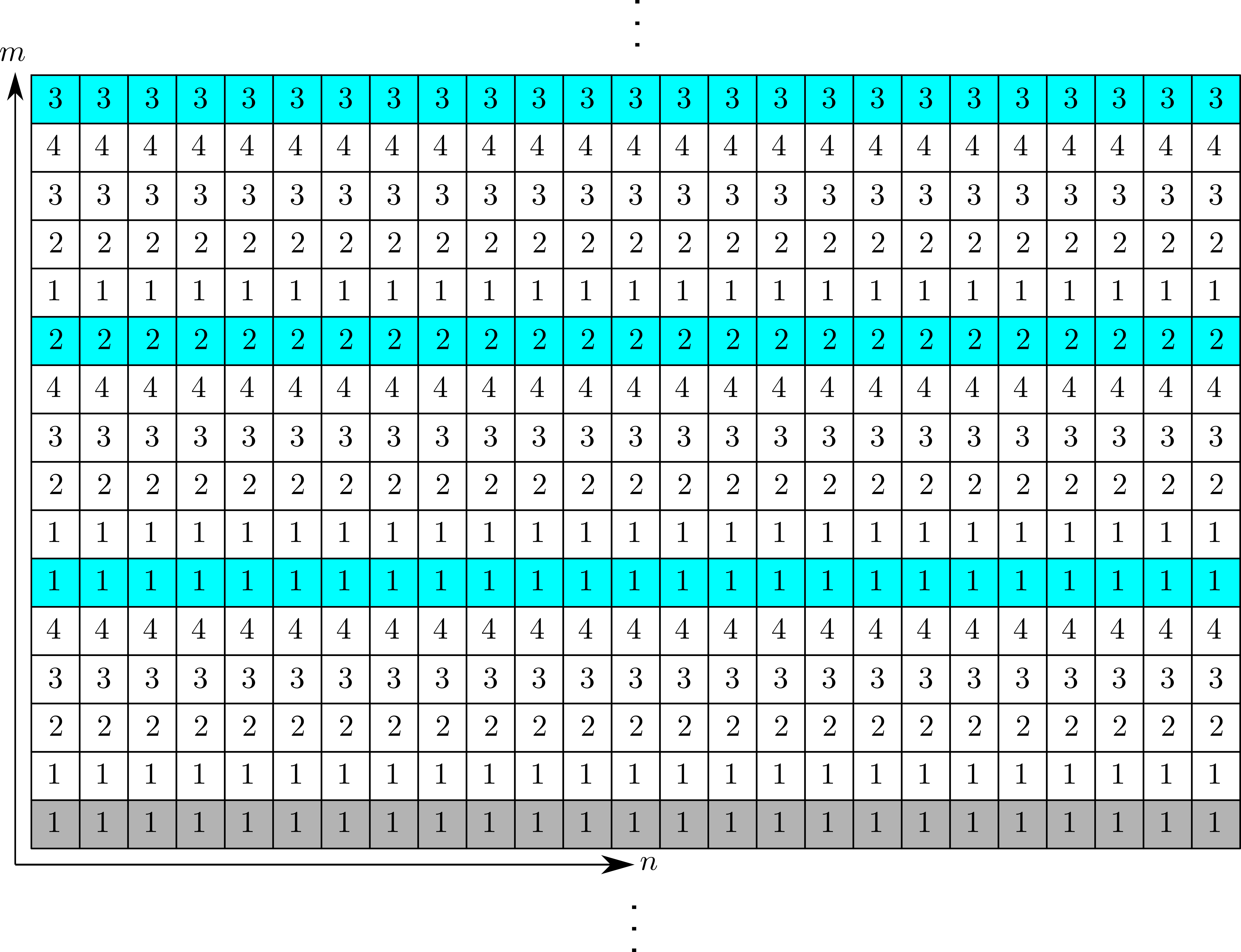}
    \caption{A horizontal slice of the $\cS_{5,25}$-Sudoku solution $F_2(n,m) \coloneqq f_5(m)$.  The colors gray, white and cyan correspond to the cases $\nu_5(m)=+\infty$, $\nu_5(m)=0$ and $\nu_5(m)=1$ respectively.}
    \label{fig:p5}
\end{figure}

The rest of the section is devoted to the proof of this theorem.  We begin with the easy direction (i).  If $A,B,C \in\Z_p$ and $F\colon \mathbb{B} \to (\Z/p\Z)^\times$ are as in (i), and $j, i$ are integers, then we have
$$ F(n,jn+i) = f_p( (A+jB) n + (C+iB) )$$
whenever $n \in \{0,\dots,N-1\}$.  If $A+jB$ and $C+iB$ both vanish, then $F(n,jn+i) = f_p(0) = f_p(1 + pn)$ for all $n \in \{0,\dots,N-1\}$, so $n \mapsto F(n,jn+i)$ lies in $\cS_{p,N}$  If instead $A+jB$ and $C+iB$ do not both vanish, then the quantity $r \coloneqq \min( \nu_p(A+jB), \nu_p(C+iB))$ is finite, and from \eqref{fp-mult} we have
$$ F(n,jn+i) = f_p\left( \frac{A+jB}{p^r} n + \frac{C+iB}{p^r} \right)$$
for all $n \in \{0,\dots,N-1\}$. If we let $a_{j,i}, b_{j,i}$ be integers with $a_{j,i} = \frac{A+jB}{p^r} \pmod{p^2}$ and $b_{j,i} = \frac{C+iB}{p^r} \pmod{p^2}$, we conclude from \eqref{fp-period} that
$$ F(n,jn+i) = f_p( a_{j,i} n + b_{j,i} )$$
whenever $n \in \{0,\dots,N-1\}$ and $\nu_p(a_{j,i} n + b_{j,i}) \leq 1$.  Comparing with Definition \ref{sudoku-construction-p} and Definition \ref{sudoku-def}, we conclude that $F$ is a $\cS_{p,N}$-Sudoku solution.

If $\nu_p(B) = 0$, then $B$ is invertible in $\Z_p$, and thus from \eqref{fp-mult} we see that
$$ F(n,m) = f_p\left(B \right) f_p\left(\frac{A}{B} n + m + \frac{C}{B}\right)$$
for all $(n,m) \in \mathbb{B}$, which easily implies that $F$ has non-constant columns as claimed.  This completes the proof of Theorem \ref{class}(i).

\subsection{Initial structure theorem}

The first main step is to establish the following initial structure theorem on $p$-adic Sudoku solutions, which will take advantage of the largeness condition \eqref{p-large}, and describes the solution outside of an exceptional set of density $1/p$.

\begin{theorem}[Initial structure theorem]\label{initial-prop}  Let $p$ be a prime obeying the largeness condition \eqref{p-large}, and let $N$ be a multiple of $p^2$.  Let $F \colon \mathbb{B} \to (\Z/p\Z)^\times$ be a $\cS_{p,N}$-Sudoku solution.  Then there exists a non-degenerate  affine form $(n,m) \mapsto A^{(0)}n+B^{(0)}m+C^{(0)}$  such that
\begin{equation}\label{fnm}
F(n,m) = \pi_p(A^{(0)}n+B^{(0)}m+C^{(0)})
 \end{equation}
whenever $(n,m) \in \mathbb{B}$ is such that $\nu_p(A^{(0)}n+B^{(0)}m+C^{(0)})=0$. Furthermore, the coefficients $A^{(0)},B^{(0)},C^{(0)}\in\Z$ are uniquely determined modulo $p$, and $F$ has non-constant columns if and only if the affine form is vertically non-degenerate modulo $p$ (i.e., if $\nu_p(B^{(0)})=0$).
\end{theorem}

We now prove this theorem, using a modification of the arguments used to establish \cite[Proposition 9.4]{GT22}.  From Definition \ref{sudoku-def} and Definition \ref{sudoku-construction-p}, and specializing to the case $\nu_p(an+b) = 0$, we see that for every non-vertical line $\ell_{j,i}$, there exists a non-degenerate  affine form $n \mapsto a_{j,i} n + b_{j,i}$, such that
\begin{equation}\label{fjn}
 F(n,jn+i) = \pi_p(a_{j,i} n + b_{j,i})
\end{equation}
for all $(n,jn+i) \in \ell_{j,i}$ with $\nu_p(a_{j,i}n + b_{j,i}) = 0$.  This is consistent with the desired conclusion \eqref{fnm}, but only gives structure on individual lines $\ell_{j,i}$ rather than the entire Sudoku board $\mathbb{B}$.  The main difficulty is then to bootstrap \eqref{fjn} to \eqref{fnm}.

In this section we will only need to use \eqref{fjn} for rows, diagonals, and anti-diagonals, thus the slope $j$ will be restricted to the range $\{-1,0,+1\}$.

We now obtain a partial version of Theorem \ref{initial-prop} in which we can obtain the desired representation \eqref{fnm} on a single square.

\begin{lemma}[Affine structure on a square]\label{aff-square}  With the notation and hypotheses of Theorem \ref{initial-prop}, there exist a square $Q_{n_0,m_0}$ for some $0 \leq n_0 \leq N-4$ and $m_0 \in \Z$ and an  affine form $(n,m) \mapsto A^{(0)}n+B^{(0)}m+C^{(0)}$ that is non-degenerate modulo $p$, such that \eqref{fnm} holds for all $(n,m) \in Q_{n_0,m_0}$.  Furthermore, $\nu_p(A^{(0)}n+B^{(0)}m+C^{(0)})=0$ for all $(n,m) \in Q_{n_0,m_0}$. 
\end{lemma}

\begin{proof}  
    For each row, diagonal, or anti-diagonal $\ell_{j,i}$, define the \emph{bad set} of that line to be the subset of $\ell_{j,i}$ given by the formula
    $$ \{ (n, jn+i) \in \ell_{j,i}: \nu_p( a_{j,i} n + b_{j,i} ) \neq 0 \}.$$
By the non-degenerate nature of the  affine form $n \mapsto a_{j,i} n + b_{j,i}$, the variable $n$ in such a bad set is restricted to at most one coset of $p\Z$  thanks to Lemma \ref{basic-affine}(i).  In particular, the cardinality of a bad set is at most $N/p$.

Call a square $Q_{n_0,m_0}$ \emph{bad} if it intersects the bad set of at least one row, diagonal, or anti-diagonal, and \emph{good} otherwise.  Intuitively, the sparsity of bad sets should lead to the sparsity of bad squares, so that good squares should exist if $p$ is large enough.  We can justify this intuition rigorously using a counting argument (basically the pigeonhole principle) as follows. Let $K$ be a large natural number, and consider all the squares $Q_{n_0,m_0}$ contained in $\{0,\dots,N-1\} \times \{1,\dots,K\}$; there are $(N-3)(K-3)$ such squares.  Next, consider all the rows, diagonals, and anti-diagonals $\ell$ that intersect $\{0,\dots,N-1\} \times \{1,\dots,K\}$; there are $3K+2(N-1)$ such lines.  As noted previously, each such line has a bad set of cardinality $N/p$.  Each element of such a bad set can make at most $4^2$ squares $Q_{n_0,m_0}$ bad (because there are $4^2$ possible cells in $Q_{n_0,m_0}$ in which that element might occur).  Putting all this together, we see that the total number of bad squares in $\{0,\dots,N-1\} \times \{1,\dots,K\}$ is at most 
    $$4^2 \frac{N}{p} (3K + 2(N-1)) \leq 48 \frac{NK}{p} + O(N^2).$$
    But with the hypothesis \eqref{p-large} (and the fact that $N \geq p^2$), we have
$$ \frac{48}{p} < 1 - \frac{3}{N}$$
and hence for $K$ large enough we have fewer than $(N-3)(K-3)$ bad squares.  Hence there must exist at least one good square $Q_{n_0,m_0}$.

Fix the good square $Q_{n_0,m_0}$.   From \eqref{fjn} and the absence of bad sets in $Q_{n_0,m_0}$, we now have the affine structure
$$ F(n, jn+i) = \pi_{p}( a_{j,i} n + b_{j,i} )$$
whenever $\ell_{j,i}$ is a row, diagonal, or anti-diagonal and $(n,jn+i)$ lies inside $Q_{n_0,m_0}$.  That is to say, $F$ is affine-linear on the restriction of any row, diagonal, or anti-diagonal to $Q_{n_0,m_0}$.  In particular, $F$ is affine on the (restricted) row
\begin{equation}\label{row}
 \ell_0=\{ (n_0+k, m_0): k = 0,1,2,3\}
 \end{equation}
and the (restricted) diagonal
\begin{equation}\label{diagonal}
\ell_1=\{ (n_0+k, m_0+k): k = 0,1,2,3\}.
\end{equation}
We may therefore find an affine function $H(n,m) = A^{(0)}n+B^{(0)}m+C^{(0)}$ for some integers $A^{(0)}, B^{(0)}, C^{(0)}$ such that $\pi_p(H)$ agrees with $F$ on both \eqref{row} and \eqref{diagonal}.  Since $F$ never vanishes, $H$ is non-degenerate modulo $p$.

The difference $F - \pi_p(H) \colon Q_{n_0,m_0} \to \Z/p\Z$ is a function which vanishes on \eqref{row} and \eqref{diagonal}, and which is affine on the restriction of any row, diagonal, or anti-diagonal to $Q_{n_0,m_0}$.  We can now show that $F- \pi_p(H)$ vanishes on the rest of $Q_{n_0,m_0}$ by the following ``Sudoku-type'' argument (followed by Figure \ref{fig:solve}), using the observation from Lemma \ref{basic-affine}(iii) that an affine function taking values in $\Z/p\Z$ that vanishes on at least two points on a line in $Q_{n_0,m_0}$, in fact vanishes identically on that line:
\begin{itemize}
    \item[(a)] The function $F - \pi_p(H)$ is affine on the restricted antidiagonal $$\{ (n_0+2-i,m_0+i): 0 \leq i \leq 2 \},$$ and is already known to vanish on two of the cells of this antidiagonal, namely $(n_0+2,m_0)$ (which lies in \eqref{row}) and $(n_0+1,m_0+1)$ (which lies in \eqref{diagonal}).  Hence it also vanishes on the third cell $(n_0,m_0+2)$.
    \item[(b)] The function $F - \pi_p(H)$ is affine on the restricted row $$\{(n_0+i, m_0+2): 0 \leq i \leq 3 \},$$ and is already known to vanish on two of the cells of this row, namely $(n_0,m_0+2)$ (which was established in (a)) and $(n_0+2,m_0+2)$ (which lies in \eqref{diagonal}.  Hence it also vanishes at the remaining two cells $(n_0+1,m_0+2), (n_0+3,m_0+2)$ of this row.
    \item[(c)]  The function $F - \pi_p(H)$ is affine on the restricted antidiagonal $$\{ (n_0+3-i,m_0+i): 0 \leq i \leq 3 \},$$ and is already known to vanish on two of the cells of this antidiagonal, namely $(n_0+1,m_0+2)$ (which was established in (b)) and $(n_0+3,m_0)$ (which lies in \eqref{row}).  Hence it also vanishes on the remaining two cells $(n_0,m_0+3)$, $(n_0+2,m_0+1)$ of this antidiagonal.
    \item[(d)]  On each restricted row $\{ (n_0+i, m_0+j): 0 \leq i \leq 3 \}$ with $j=1,2,3$, the  function $F - \pi_p(H)$ is affine and is already known to vanish on at least two of the cells, so vanishes identically.  
\end{itemize}
 Thus $F-\pi_p(H)$ vanishes on all of $Q_{n_0,m_0}$, and the claim follows.

\begin{figure}
    \centering
    \includegraphics[width = .7\textwidth]{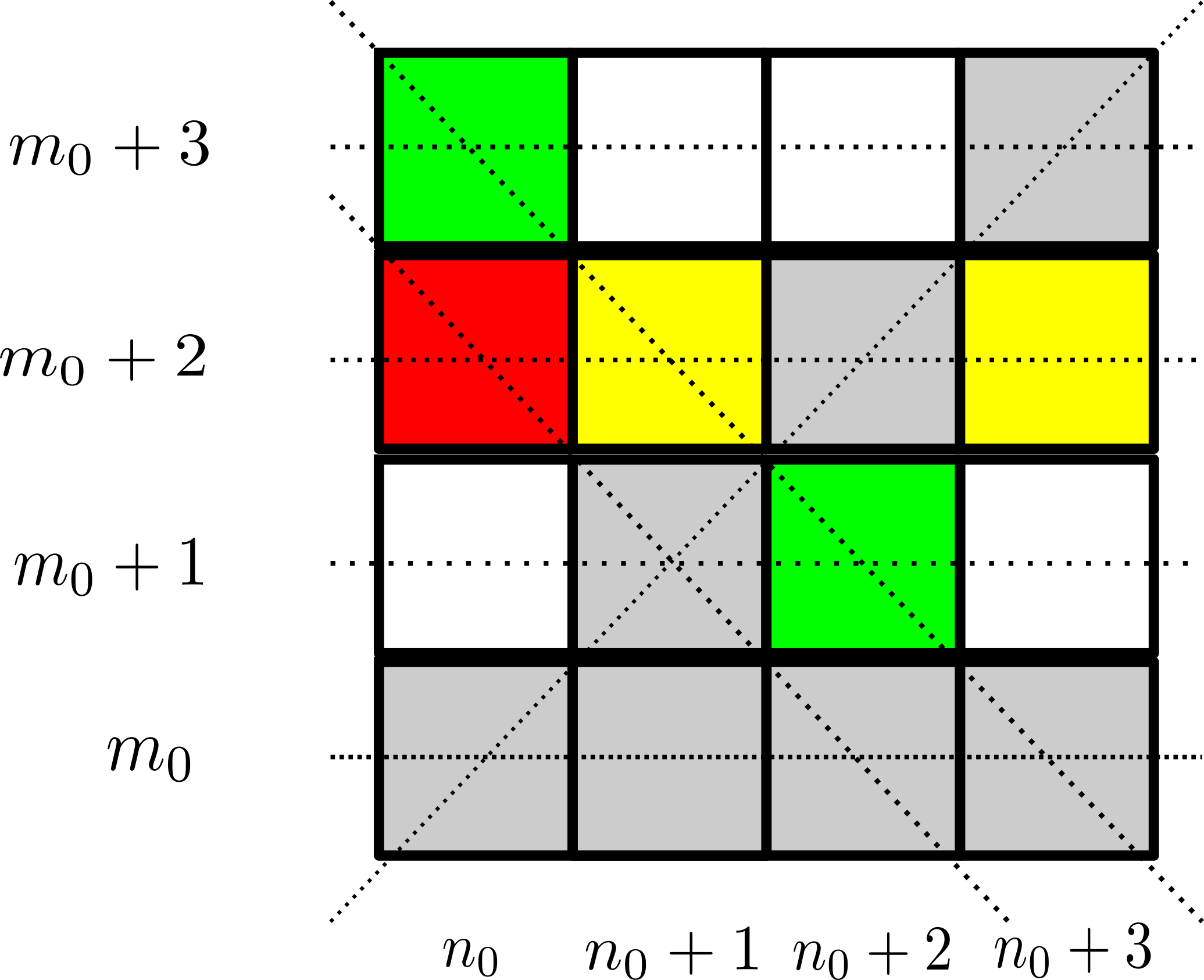}
    \caption{The Sudoku ``puzzle'' on $Q_{n_0,m_0}$ is ``solved'' in the following order.  Firstly, $F$ is known to equal $\pi_p(H)$ on the row \eqref{row} and diagonal \eqref{diagonal} (colored in grey). 
 In (a), using an antidiagonal the identity $F=\pi_p(H)$ extended to the red cell; in (b), using a row the identity is then extended to the yellow cells, and then, in (c), using another antidiagonal it is extended to the green cells, and then finally, in (d), using two rows it is extended to the remaining white cells of $Q_{n_0,m_0}$.}
    \label{fig:solve}
\end{figure}
\end{proof}

Now we can complete the proof of Theorem \ref{initial-prop}.  Let $Q_{n_0,m_0}, A^{(0)}, B^{(0)}, C^{(0)}$ be as in Lemma \ref{aff-square}.  Call a cell $(n,m) \in \mathbb{B}$ \emph{good} if at least one of the following assertions holds:
\begin{itemize}
    \item[(a)]  $\nu_p(A^{(0)}n+B^{(0)}m+C^{(0)}) \neq 0$.
    \item[(b)]  $F(n,m) = \pi_{p}(A^{(0)}n+B^{(0)}m+C^{(0)})$.
\end{itemize}
From Lemma \ref{aff-square} we see that all sixteen cells of $Q_{n_0,m_0}$ are good.

We now make the following key claim, which we call the \emph{extension property}: if $\ell_{j,i}$ is a non-vertical line that contains four consecutive good cells, then in fact all the cells in $\ell_{j,i}$ are good.  We may assume that the  affine form 
$$n \mapsto A^{(0)}n + B^{(0)}(jn+i) + C^{(0)}$$ is non-degenerate modulo $p$, since we automatically have (a) otherwise.  By Lemma \ref{basic-affine}(i) we have $\nu_p(A^{(0)}n + B^{(0)}(jn+i) + C^{(0)})=0$ outside of at most one coset of $p\Z$, which contains at most one of the four consecutive good cells thanks to \eqref{p-large}.  On the other hand, from \eqref{fjn} and Lemma \ref{basic-affine}(i) we also know $F(n,jn+i)$ must also agree with an affine function $\pi_{p}( a_{j,i} n + b_{j,i} )$ for $n \in \{0,\dots,N-1\}$ outside of a (potentially different) coset of $p\Z$, which again can contain at most one of the four consecutive good cells.  Thus the two  affine forms $A^{(0)}n + B^{(0)}(jn+i) + C^{(0)}$ and  $a_{j,i} n + b_{j,i}$ must agree modulo $p$ on at least two of the four consecutive good cells. By Lemma \ref{basic-affine}(iii) (and \eqref{p-large}), this implies that all the cells in $\ell_{j,i}$ are good, as claimed.

We now use the extension property to greatly expand the set of cells that are known to be good (somewhat analogously to how one would solve a Sudoku puzzle).  Firstly, for each of the four rows $\{0,\dots,N-1\} \times \{m_0+i\}$ for $0 \leq i \leq 3$, we already know, by Lemma \ref{aff-square}, that these rows contain four  consecutive good cells.  Thus, by the extension property, all the cells of these rows are good. In other words, the $N \times 4$ block
\begin{equation}\label{block}
 \{0,\dots,N-1\} \times \{m_0+i: 0 \leq i \leq 3\}
 \end{equation}
of four consecutive rows consists entirely of good cells.

We now claim that the next row $\{0,\dots,N-1\} \times \{m_0+4\}$ also consists entirely of good cells.  Indeed, for any $N-4 \leq i \leq N-1$, the diagonal line $\{ (n, n-i+m_0+4): 1 \leq n \leq N \}$ passes through four consecutive good cells of the block \eqref{block}, followed by the cell $(i, m_0+4)$.  Thus, by the extension property, this latter cell $(i,m_0+4)$ is also good.  This establishes four consecutive good cells in the row $\{0,\dots,N-1\} \times \{m_0+4\}$, and so by another appeal to the extension property we conclude that this row also consists entirely of good cells. 

By Proposition \ref{sudoku-affine}(i), the horizontal reflection $(n,m) \mapsto F(n,-m)$ of $F$ is again a $\cS_{p,N}$-Sudoku solution.  Applying this reflection symmetry to the above arguments, we also conclude that all the cells of the row $\{0,\dots,N-1\} \times \{m_0-1\}$ immediately below the block \eqref{block} are also good.  By inductively extending the block \eqref{block} one row at a time in both directions, we thus conclude that the entire Sudoku board ${\mathbb B}$ is good. See an illustration of this extension argument in Figure \ref{fig:extend}.

\begin{figure}
    \centering
    \includegraphics[width = .9\textwidth]{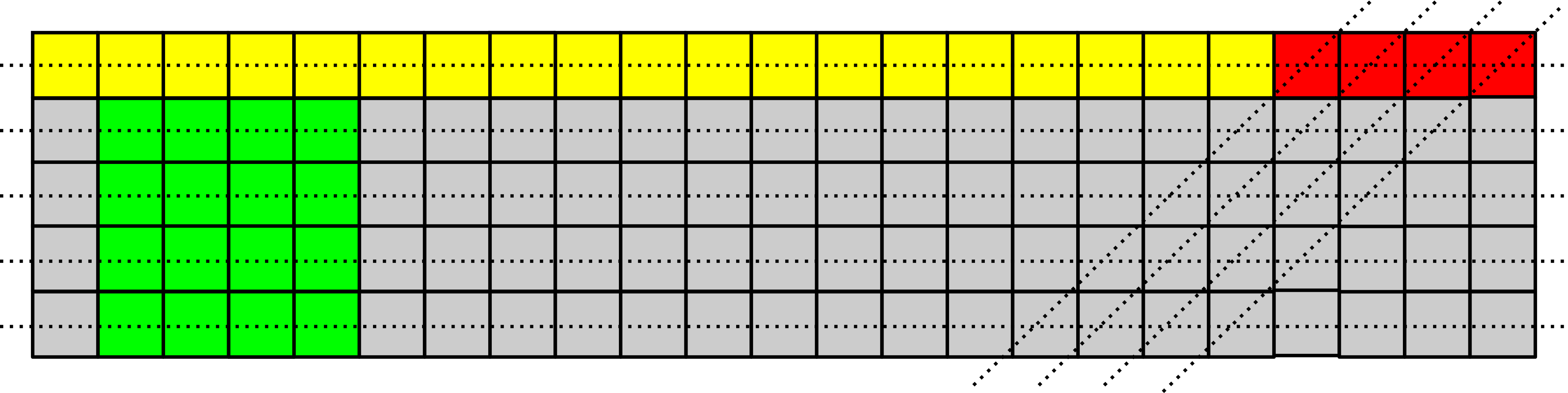}
    \caption{Once one knows that all the cells in the green square are good, one can use the extension property along rows to also establish that the grey cells are good.  Applying the extension property along the four indicated diagonals, one then obtains that the red cells are good, and then by applying the extension property along a row one then obtains that the remaining cells of the row above the original square (yellow) are good.  A reflection of this argument also handles establishes that the row below the square is good; iterating these arguments allows one to show that all the cells on the Sudoku board are good.}
    \label{fig:extend}
\end{figure}

This establishes Theorem \ref{initial-prop} except for the uniqueness claim, which we will prove next.  Suppose that the conclusions of the theorem held for two non-degenerate  affine forms $(n,m) \mapsto A^{(0)}n+B^{(0)}m+C^{(0)}$ and $(n,m) \mapsto \tilde A^{(0)}n+\tilde B^{(0)}m+\tilde C^{(0)}$.  Then by \eqref{fnm}, Lemma \ref{basic-affine}(ii), we see that $A^{(0)}n+B^{(0)}m+C^{(0)}$ and $\tilde A^{(0)}n+\tilde B^{(0)}m+\tilde C^{(0)}$ would agree modulo $p$ outside of at most $2p$ cosets of $p\Z\times p\Z$, and hence must be identical modulo $p$ by Lemma \ref{basic-affine}(iv) (and \eqref{p-large}), giving the uniqueness claim.  Finally, it is clear from \eqref{fnm} that $F$ has non-constant columns if $\nu_p(B^{(0)})=0$, and will have constant columns for all $n$ outside of at most one coset of $p\Z$ if $\nu_p(B^{(0)}) > 0$.  This completes the proof of Theorem \ref{initial-prop}.

\subsection{Intermediate structure theorem}

We now upgrade Theorem \ref{initial-prop} (under the assumption of non-constant columns), reducing the density of the exceptional set on which the solution remains uncontrolled from $1/p$ to $1/p^{r+1}$ for a given $r$.

\begin{theorem}[Intermediate structure theorem]\label{final-prop}  Let $p$ be a prime obeying the largeness condition
\eqref{p-large}, and let $N$ be a multiple of $p^2$.  Let $F \colon \mathbb{B} \to (\Z/p\Z)^\times$ be a $\cS_{p,N}$-Sudoku solution with non-constant columns.  Then for any $r \geq 0$, there exists an affine form $(n,m) \mapsto A^{(r)}n+B^{(r)}m+C^{(r)}$ that is vertically non-degenerate modulo $p$, and such that
\begin{equation}\label{fnm-f}
 F(n,m) = f_p(A^{(r)}n+B^{(r)}m+C^{(r)})
 \end{equation}
whenever $(n,m) \in \mathbb{B}$ is such that $\nu_p(A^{(r)}n+B^{(r)}m+C^{(r)}) \leq r$.
\end{theorem}

We prove this theorem by induction on $r$.  The case $r=0$ is immediate from Theorem \ref{initial-prop}.  We now turn to the $r=1$ case, which requires a special argument.

Let $F$ be a $\cS_{p,N}$-Sudoku solution with non-constant columns.   By Theorem \ref{initial-prop} and the non-constant columns hypothesis, we can find an  affine form $(n,m) \mapsto A^{(0)} n + B^{(0)} m + C^{(0)}$ with $\nu_p(B^{(0)})=0$, such that
$$ F(n,m) = \pi_p(A^{(0)}n+B^{(0)}m+C^{(0)})$$
whenever $(n,m) \in \mathbb{B}$ with $\nu_p(A^{(0)}n+B^{(0)}m+C^{(0)}) = 0$.

We can eliminate $A^{(0)}, C^{(0)}$ by the following device.  As $B^{(0)}$ is invertible in $\Z/p\Z$ (and hence in $\Z_p$), we can find integers $D, E$ such that the  affine forms $(n,m) \mapsto A^{(0)}n+B^{(0)}m+C^{(0)}$ and $(n,m) \mapsto B^{(0)} ( m - Dn - E)$ are identical modulo $p$.  From \eqref{fp-mult}, \eqref{fp-period} we then have
$$ F(n,m) = \pi_{p}(B^{(0)}(m-Dn-E))$$
whenever $(n,m) \in \mathbb{B}$ with $\nu_p( m-Dn-E ) = 0$.  By replacing $F$ with the sheared function $(n,m) \mapsto F(n,m+Dn+E)$ (which remains a ${\cS_{p,N}}$-Sudoku solution thanks to Proposition \ref{sudoku-affine}(ii), and also continues to have non-constant columns), we may assume without loss of generality that $D=E=0$, thus we have
\begin{equation}\label{fnm-00}
 F(n,m) = \pi_p(B^{(0)} m)
\end{equation}
whenever $(n,m) \in \mathbb{B}$ with $\nu_p( m ) = 0$.  With a similar argument it is also possible to normalize $B^{(0)}=1$, but we will not do so here as it does not provide significant simplification to the arguments below.

To establish the $r\geq 1$ cases of Theorem \ref{final-prop}, we need understand the structure of $F(n,m)$ when $m$ is divisible by $p$.  We apply a ``Tetris move'' by considering the function $(n,m) \mapsto F(n,pm)$, which is also a $\cS_{p,N}$-Sudoku solution thanks to Proposition \ref{sudoku-affine}(iii).  Thus by another appeal to Theorem \ref{initial-prop}, we may find an  affine form $(n,m) \mapsto A_1 n + B_1 m + C_1$ that is non-degenerate modulo $p$, such that
\begin{equation}\label{fnm-10}
 F(n,p m) = \pi_p(A_1n+B_1m+C_1)
\end{equation}
whenever $(n,m) \in \mathbb{B}$ with $\nu_p( A_1n+B_1m+C_1 ) = 0$.

To proceed further we will need to establish some compatibility conditions between the  affine forms $(n,m) \mapsto B^{(0)} m$ and $(n,m) \mapsto A_1 n + B_1 m + C_1$.  To do this, we will study the Sudoku on various non-vertical lines $\ell_{j,i}$ whose slope $j$ is positive but lies outside of $p\Z$ (so that $\nu_p(j)=0$). By Definition \ref{sudoku-def} and Definition \ref{sudoku-construction-p}, we see that for any such line $\ell_{j,i}$, there exists an  affine form $n \mapsto a_{j,i} n + b_{j,i}$ that is non-degenerate modulo $p$  such that
\begin{equation}\label{f-on-line}
F(n,jn+i) = f_p( a_{j,i} n + b_{j,i} )
\end{equation}
whenever $n \in \{0,\dots,N-1\}$ is such that $ \nu_p( a_{j,i} n + b_{j,i} ) \leq 1$.  In particular, we have
$$ F(n,jn+i) = \pi_p( a_{j,i} n + b_{j,i} )$$
whenever $n \in \{0,\dots,N-1\}$ is such that $\nu_p(a_{j,i} n + b_{j,i}) = 0$.  Comparing this with \eqref{fnm-10}, we conclude that $B^{(0)}(jn+i)$ and $a_{j,i} n+b_{j,i}$ agree modulo $p$ in $n \in \{0,\dots,N-1\}$ outside of at most two cosets of $p\Z$, and thus are identical modulo $p$ thanks to Lemma \ref{basic-affine}(iii) (and \eqref{p-large}).  Thus we have
\begin{equation}\label{aji}
 a_{j,i} = B^{(0)} j \pmod{p}
 \end{equation}
and
\begin{equation}\label{bji}
b_{j,i} = B^{(0)} i \pmod{p}
\end{equation}
In particular, $a_{j,i}$ is coprime to $p$.

We now specialize \eqref{f-on-line} to the case $\nu_{p}(a_{j,i} n + b_{j,i}) = 1$, obtaining
\begin{equation}\label{f-on-line-3}
F(n,jn+i) = \pi_p\left( \frac{a_{j,i} n + b_{j,i}}{p} \right)
\end{equation}
for all $n \in \{0,\dots,N-1\}$ with $\nu_{p}(a_{j,i} n + b_{j,i}) = 1$.   Meanwhile, from \eqref{fnm-10} we have
\begin{equation}\label{f-on-line-4}
 F(n,jn+i) = \pi_p\left(A_1n+B_1\frac{jn+i}{p}+C_1\right)
\end{equation}
whenever $n \in \{0,\dots,N-1\}$ is such that $p$ divides $jn+i$, and $\nu_p(A_1n+B_1 \frac{jn+i}{p}+C_1)=0$.

To compare \eqref{f-on-line-3} with \eqref{f-on-line-4} it is convenient to make the change of variables
$$ n = \frac{p m - i}{j}$$ 
where $m$ now ranges in the set
$$ P \coloneqq \left\{ m \in \Z: \frac{i}{p} \leq m \leq \frac{j(N-1)+i}{p}; p m = i \pmod{j} \right\},$$
which is an arithmetic progression of spacing $j$ and length $N/p$.  Then we may write
$$ a_{j,i} n + b_{j,i} = p( a_{j,i} m + b'_{j,i} ) / j$$
where
$$ b'_{j,i} \coloneqq \frac{b_{j,i} j - a_{j,i} i}{p}$$
(which is an integer thanks to \eqref{aji}, \eqref{bji}).  From \eqref{f-on-line-3} we thus have
\begin{equation}\label{f-on-line-3a}
F(\frac{pm-i}{j},pm) = \pi_p( ( a_{j,i} m + b'_{j,i} ) / j )
\end{equation}
for all $m \in P$ with $\nu_p( a_{j,i} m + b'_{j,i} ) = 0$, while from \eqref{f-on-line-4} we have
\begin{equation}\label{f-on-line-4a}
F(\frac{pm-i}{j},pm) = \pi_p( (A_1 (pm-i) + B_1 jm + C_1 j)/j )
\end{equation}
for all $m \in P$ with $\nu_p(A_1 (pm-i) + B_1 jm + C_1 j ) = 0$.

We would like the  affine form 
\begin{equation}\label{maf}
m \mapsto A_1 (p m - i) + B_1 j m + C_1 j
\end{equation}
to be non-degenerate modulo $p$.  Extracting coefficients, we see that this is the case unless
\begin{equation}\label{aoi}
 B_1 = -A_1 i + C_1 j = 0 \pmod{p}.
\end{equation}
But as $(n,m) \mapsto A_1 n + B_1 m + C_1$ is non-degenerate mod $p$, we see from Lemma \ref{basic-affine}(ii) that \eqref{aoi} can only occur for $(j,i)$ in at most $p$ cosets of $p\Z\times p\Z$.  Thus we see for all $(j,i) \in \N \times \Z$ outside of at most $2p$ cosets of $p\Z\times p\Z$, $j$ is coprime to $p$ and that the  affine form \eqref{maf} is non-degenerate mod $p$.

Assuming the non-degeneracy of \eqref{maf}, we now see from \eqref{f-on-line-3a} and \eqref{f-on-line-4a} that one has the two  affine forms $$m \mapsto a_{j,i} m + b'_{j,i},\quad m \mapsto A_1 (p m - i) + B_1 j m + C_1 j$$ agree modulo $p$ outside of at most two cosets of $p\Z$, and thus are identical mod $p$ thanks to Lemma \ref{basic-affine}(iii) (and \eqref{p-large}) to conclude in particular (using \eqref{aji}) that
\begin{equation}\label{bob}
 B_1 = a_{j,i}/j = B^{(0)}  \pmod{p}.
 \end{equation}
If we now set
$$ A^{(1)} \coloneqq p A_1; \quad B^{(1)} \coloneqq B_1; \quad C^{(1)} \coloneqq p C_1$$
then $B^{(1)}$ is coprime to $p$, so the  affine form $(n,m) \mapsto A^{(1)}n+B^{(1)}m+C^{(1)}$ is vertically non-degenerate modulo $p$.  If $(n,m) \in \mathbb{B}$ is such that $\nu_p(A^{(1)}n+B^{(1)}m+C^{(1)})=0$ then $\nu_p(m)=0$, and then by \eqref{fnm-00} and \eqref{bob} we have
$$ F(n,m) = \pi_p( B^{(0)} m ) = \pi_p(B^{(1)} m) = f_p(A^{(1)}n+B^{(1)}m+C^{(1)}).$$
If instead $(n,m) \in \mathbb{B}$ is such that $\nu_p(A^{(1)}n+B^{(1)}m+C^{(1)})=1$, then 
$m=pm'$ for some integer $m'$, and $A^{(1)}n+B^{(1)}m+C^{(1)} = p(A_1 n + B_1 m' + C_1)$, thus $\nu_p(A_1 n + B_1 m' + C_1)=0$.  From \eqref{fnm-10} we then have
$$    F(n,m) = \pi_p(A_1n+B_1m+C_1) = f_p(A^{(1)}n + B^{(1)}m + C^{(1)}).$$
This establishes the $r=1$ case of Theorem \ref{final-prop}.

We record a useful consequence of this $r=1$ case:

\begin{corollary}\label{nonconst-dil}  Let $p$ be a prime obeying the largeness condition
\eqref{p-large}, and let $N$ be a multiple of $p^2$.  If $F \colon \mathbb{B} \to (\Z/p\Z)^\times$ is a $\cS_{p,N}$-Sudoku solution with non-constant columns, then $(n,m) \mapsto F(n,pm)$ is also a 
$\cS_{p,N}$-Sudoku solution with non-constant columns.
\end{corollary}

\begin{proof} 
The fact that  $(n,m) \mapsto F(n,pm)$ is a $\cS_{p,N}$-Sudoku solution already follows from Proposition \ref{sudoku-affine}(iii).  By the $r=1$ case of Theorem \ref{final-prop}, we can find an  affine form $(n,m) \mapsto A^{(1)}n+B^{(1)}m+C^{(1)}$ with $B^{(1)}$ coprime to $p$ such that
$$ F(n,m) = f_p(A^{(1)}n+B^{(1)}m+C^{(1)}) $$
for all $(n,m) \in \mathbb{B}$ with $A^{(1)}n+B^{(1)}m+C^{(1)}$ not divisible by $p^2$.  From this it is easy to see that the function $m \mapsto F(n,pm)$ is non-constant for any $n=1,\dots,N$, so we obtain the required non-constant column condition.
\end{proof}

Finally, suppose for inductive purposes Theorem \ref{final-prop} is already established for some $r \geq 1$.
If  $F \colon \mathbb{B} \to \Sigma$ is a $\cS_{p,N}$-Sudoku solution with non-constant columns, then by hypothesis there exists an  affine form $(n,m) \mapsto A^{(r)}n+B^{(r)}m+C^{(r)}$ that is vertically non-degenerate modulo $p$, such that
$$
F(n,m) = f_p(A^{(r)}n+B^{(r)}m+C^{(r)})
$$
whenever $(n,m) \in \mathbb{B}$ is such that $\nu_p(A^{(r)}n+B^{(r)}m+C^{(r)}) \leq r$.  As in the $r=1$ case, we may find integers $D, E$ such that the  affine forms $(n,m) \mapsto A^{(r)} n + B^{(r)} m + C^{(r)}$ and $(n,m) \mapsto B^{(r)} (m - Dn - E)$ are identical modulo $p^r$.  We may then apply a shear transformation $(n,m) \mapsto (n,m+Dn+E)$ as before to assume without loss of generality that
\begin{equation}\label{nup-base}
F(n,m) = f_p(B^{(r)}m)
\end{equation}
whenever $(n,m) \in \mathbb{B}$ is such that $\nu_p(m) \leq r$, thus effectively setting $A^{(r)}$ and $C^{(r)}$ to zero.

By Corollary \ref{nonconst-dil}, the function $(n,m)\mapsto F(n,pm)$ is a $\cS_{p,N}$-Sudoku solution with non-constant columns. Thus, by the induction hypothesis, we can find an  affine form $(n,m) \mapsto A_{r+1}n+B_{r+1}m+C_{r+1}$ that is vertically non-degenerate modulo $p$ such that
\begin{equation}\label{fnpm}
F(n,pm) = f_p(A_{r+1} n + B_{r+1} m + C_{r+1})
\end{equation}
whenever $(n,m) \in \mathbb{B}$ is such that $\nu_p(A_{r+1} n + B_{r+1} m + C_{r+1}) \leq r$. 
In particular, by combining \eqref{fnpm} with \eqref{nup-base} we have
$$
F(n,p^r m) =  f_p(A_{r+1} n + B_{r+1} p^{r-1} m + C_{r+1}) = f_p(B^{(r)} m)$$
whenever $(n,m) \in \mathbb{B}$ is such that $\nu_p(m)=0$ and $\nu_p(A_{r+1} n + B_{r+1} p^{r-1} m + C_{r+1}) \leq r$.  Since $\nu_p(B_{r+1})=0$, we may  factor 
$$A_{r+1} n + B_{r+1} p^{r-1} m + C_{r+1} = p^j (A'_{r+1} n + B'_{r+1} m + C'_{r+1})$$ 
for some $0 \leq j \leq r-1$ and some  affine form $(n,m) \mapsto A'_{r+1}n+B'_{r+1}m+C'_{r+1}$ that is non-degenerate modulo $p$.  We conclude that
$$\pi_p( A'_{r+1} n + B'_{r+1} m + C'_{r+1} ) = \pi_p( B^{(r)} m )$$
whenever $(n,m) \in \mathbb{B}$ is such that 
\begin{equation}\label{nup0}
\nu_p(A'_{r+1}n+B'_{r+1}m+C'_{r+1}) = \nu_p(B^{(r)} m)=0.
\end{equation}
Since \eqref{nup0} holds outside of at most $2p$ cosets of $p\Z\times p\Z$, we see from Lemma \ref{basic-affine}(iv) and \eqref{p-large} that $A'_{r+1}n+B'_{r+1}m+C'_{r+1}$ and $B^{(r)} m$ are identical mod $p$, so $B'_{r+1} = B_{r+1} p^{r-1-j}$ is coprime to $p$ and $A'_{r+1}, C'_{r+1}$ are divisible by $p$. This forces $j=r-1$, and $A_{r+1}, C_{r+1}$ divisible by $p^r$, and also $B_{r+1} = B^{(r)} \pmod{p}$.  We then claim that
$$ F(n,m) = f_p( p A_{r+1} n + B_{r+1} m + p C_{r+1} )$$
whenever $(n,m) \in \mathbb{B}$ is such that $\nu_p( p A_{r+1} n + B_{r+1} m + p C_{r+1} ) \leq r+1$.  When $m$ is divisible by $p$ this follows from \eqref{fnpm} and \eqref{fp-mult}, while for $m$ coprime to $p$ this follows from \eqref{nup-base} and \eqref{fp-period}.  Since the form $(n,m) \mapsto pA_{r+1} n + B_{r+1} m + pC_{r+1}$ is vertically non-degenerate modulo $p$, by setting
$$A^{(r+1)}\coloneqq pA_{r+1},\quad B^{(r+1)}\coloneqq B_{r+1},\quad C^{(r+1)}\coloneqq pC_{r+1},$$ 
we conclude that Theorem \ref{final-prop} holds for $r+1$,  closing the induction, and completing the proof of Theorem \ref{final-prop} for all $r \geq 0$.

We are now ready to establish Theorem \ref{class}(ii). 

\subsection{Proof of Theorem \ref{class}(ii)}
 Let $F$ be a $\cS_{p,N}$-Sudoku solution with non-constant columns.  By Theorem \ref{final-prop} (and the axiom of choice), we see that for each $r \geq 0$ we can find an  affine form $(n,m) \mapsto A^{(r)} n + B^{(r)} m + C^{(r)}$ with $\nu_p(B^{(r)})=0$ such that
$$ F(n,m) = f_p(A^{(r)} n + B^{(r)} m + C^{(r)})$$
whenever $(n,m) \in \mathbb{B}$ is such that $\nu_p(A^{(r)} n + B^{(r)} m + C^{(r)}) \leq r$.  By the compactness of $\Z_p$, we can find a subsequence $(A^{(r_i)}, B^{(r_i)}, C^{(r_i)})$ of $(A^{(r)},B^{(r)},C^{(r)})$ which converges in the $p$-adic topology to some limit $(A,B,C) \in \Z_p^3$.  It is then a routine matter to pass to the limit (using the continuity properties of $f_p$ and $\nu_p$ arising from \eqref{fp-period}) and conclude that
$$ F(n,m) = f_p(A n + B m + C)$$
whenever $(n,m) \in \mathbb{B}$ is such that $An+Bm+C \neq 0$, giving Theorem \ref{class}(ii).

\begin{figure}
    \centering
    \includegraphics[width = .8\textwidth]{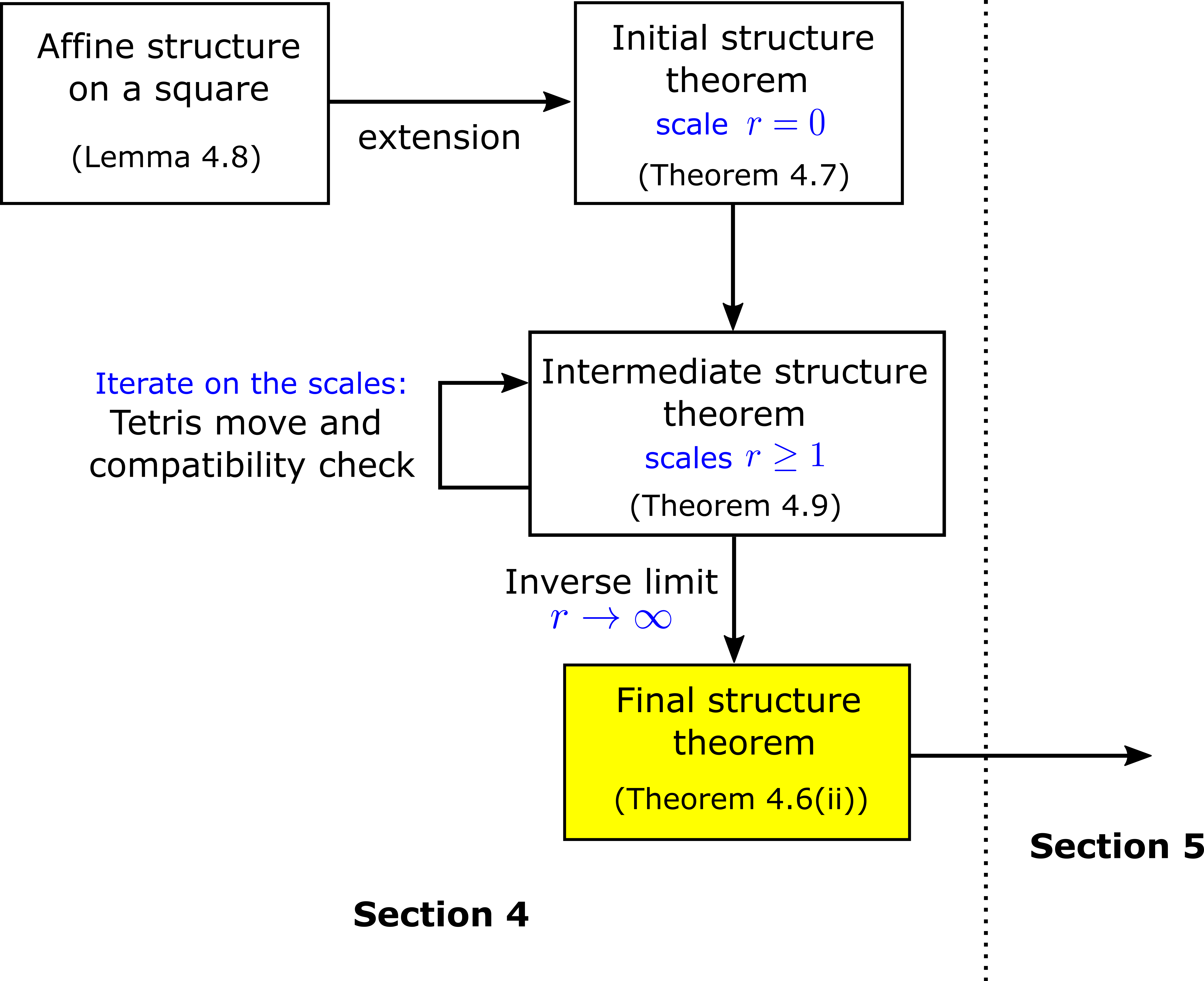}
    \caption{A high level overview of the proof of Theorem \ref{class}(ii) in Section \ref{sec:4}. This will be used in Section \ref{sec:5} to prove Theorem \ref{domino-to-sudoku}; see Figure \ref{fig:sec4-5}.}
    \label{fig:sec4}
\end{figure}

We summarize the steps just taken to establish Theorem \ref{class}(ii) in Figure \ref{fig:sec4}. 

\begin{remark}\label{alt}  We can use \eqref{fp-mult} to write the conclusion of Theorem \ref{class}(ii) in the alternate form
\begin{equation}\label{fnm-cDE}
F(n,m) = c f_p(m - Dn - E)
\end{equation}
whenever $(n,m) \in \mathbb{B}$ is such that $m \neq Dn+E$, where $c \in (\Z/p\Z)^\times$ and $D,E \in \Z_p$ are given by the formulae
$$ c \coloneqq f_p(B); \quad D \coloneqq - \frac{A}{B}; \quad E \coloneqq - \frac{C}{B}.$$
In this form one can show that the coefficients $c,D,E$ are uniquely determined by $F$ (thus giving a one-to-one correspondence between $\cS_{p,N}$-Sudoku solutions with non-constant columns and triples $(c,D,E)$ in $(\Z/p\Z)^\times \times \Z_p \times \Z_p$).  Indeed, from \eqref{fnm-cDE} we have
$$ F(n,m+1) - F(n,m) = c$$
whenever $m, m+1 \neq Dn+E$, which ensures that $c$ is uniquely determined by $F$.  It then suffices to show that $\pi_{p^r}(D), \pi_{p^r}(E)$ is uniquely determined by $F$ for any $r \geq 0$.  This is trivial for $r=0$. Assuming inductively that $r \geq 1$ and that $\pi_{p^{r-1}}(D), \pi_{p^{r-1}}(E)$ has already been shown to be uniquely determined by $F$, we see from \eqref{fnm-cDE} that for any $(n,m) \in \mathbb{B}$ with $m = Dn + E \pmod{p^{r-1}}$, we have
$$ F(n,m+p^{r-1} h) = c \pi_p( h - h_{n,m} )$$
whenever $h \in \Z$ and $\pi_p(h) \neq \pi_p(h_{n,m})$, where $h_{n,m} \coloneqq \frac{m-Dn-E}{p^{r-1}}$.  From this (and Lemma \ref{basic-affine}(iii)) we see that $\pi_p(h_{n,m})$ is uniquely determined by $n,m$, which implies that $\pi_{p^r}(D), \pi_{p^r}(E)$ are uniquely determined by $F$, as desired.
\end{remark}

\section{From the domino problem to a decorated $p_1\times p_2$-adic Sudoku puzzle}\label{sec:5}

In this section we will prove Theorem \ref{domino-to-sudoku}.  We will need two distinct primes $p_1,p_2$ obeying \eqref{p-large}; for instance, we can take
$$ p_1 = 53; \quad p_2 = 59$$
although the precise values of these primes will not be of significance to our arguments as long as they satisfy the largeness condition \eqref{p-large}.  We also set the width $N$ to be
$$ N \coloneqq p_1^2 p_2^2,$$
so that $N$ is a multiple of both $p_1^2$ and $p_2^2$.  We adopt the abbreviations
$$ f_{p_1,p_2}(n) \coloneqq (f_{p_1}(n), f_{p_2}(n)); \quad \nu_{p_1,p_2}(n) \coloneqq (\nu_{p_1}(n), \nu_{p_2}(n)).$$

Now we assign a ``decorated $p_1\times p_2$-adic Sudoku'' rule to any domino set.

\begin{definition}[Constructing a decorated $p_1\times p_2$-adic Sudoku rule]\label{sudoku-construction} 
Let $\cR = (\cW, \cR_1, \cR_2)$
be a domino set.  We then construct a Sudoku rule $\cS^\cR = \cS^\cR_{N,\Sigma}$ as follows.
\begin{itemize}
    \item[(i)] We set the width $N$ of the Sudoku to be
    $$ N \coloneqq p_1^2 p_2^2.$$
    \item[(ii)] We set the digit set $\Sigma$ of the Sudoku to be
    $$ \Sigma \coloneqq (\Z/p_1\Z)^\times \times (\Z/p_2\Z)^\times \times \cW.$$
    \item[(iii)] We set $\cS^\cR$ to be the collection of all functions $g \colon \{0,\dots,N-1\} \to \Sigma$ of the form $g = (g_1, g_2, w)$, with 
    \begin{align*}
        g_1 &\colon \{0,\dots,N-1\} \to (\Z/p_1\Z)^\times\\
        g_2 &\colon \{0,\dots,N-1\} \to (\Z/p_2\Z)^\times\\
        w &\colon \{0,\dots,N-1\} \to \cW,
    \end{align*}
     for which there exist an affine form $n \mapsto a n + b$, non-degenerate modulo $p_1$ as well as modulo $p_2$, together with a $\cR$-domino function  
     $$\cT \colon [(0,0),(t_1,t_2)] \to \cW$$ on the rectangle $[(0,0),(t_1,t_2)]$, with $t_l$ for $l=1,2$ defined to equal $1$ when $\nu_{p_l}(a) = 0$ and $0$ otherwise, such that for $n=1,\dots,N$, one has
    \begin{align*}
    g_l(n) &=  f_{p_l}(an+b) \hbox{ whenever } l=1,2 \hbox{ and } \nu_{p_l}(an+b) \leq t_l; \\
    w(n) &= \cT( \nu_{p_1,p_2}(an+b) ) \hbox{ whenever } \nu_{p_1,p_2}(an+b) \leq (t_1,t_2).
    \end{align*}
\end{itemize}
\end{definition}

We write a $\cS^\cR$-Sudoku solution $F$ as a triple $F = (F_1, F_2, w)$, where 
\begin{align*}
    F_1 &\colon \mathbb{B} \to (\Z/p_1\Z)^\times\\
    F_2 &\colon \mathbb{B} \to (\Z/p_2\Z)^\times\\
    w &\colon \mathbb{B} \to \cW.
\end{align*}
From the definitions we see that if $F = (F_1,F_2,w)$ is a $\cS^\cR$-Sudoku solution, then 
$F_l$ is a $\cS_{p_l,N}$-Sudoku solution for $l=1,2$.

The key proposition is then

\begin{proposition}\label{dom-solv}  Let $p_1, p_2$ be distinct primes obeying \eqref{p-large}, let $N \coloneqq p_1^2 p_2^2$, and let $\cR$ be a domino set.  Then the $\cR$-domino problem is solvable on $\N^2$ if and only if there exists a $\cS^\cR$-Sudoku solution $F = (F_1,F_2, w)$ in which $F_1, F_2$ both have non-constant columns.
\end{proposition}

\begin{figure}
    \centering
    \includegraphics[width = .9\textwidth]{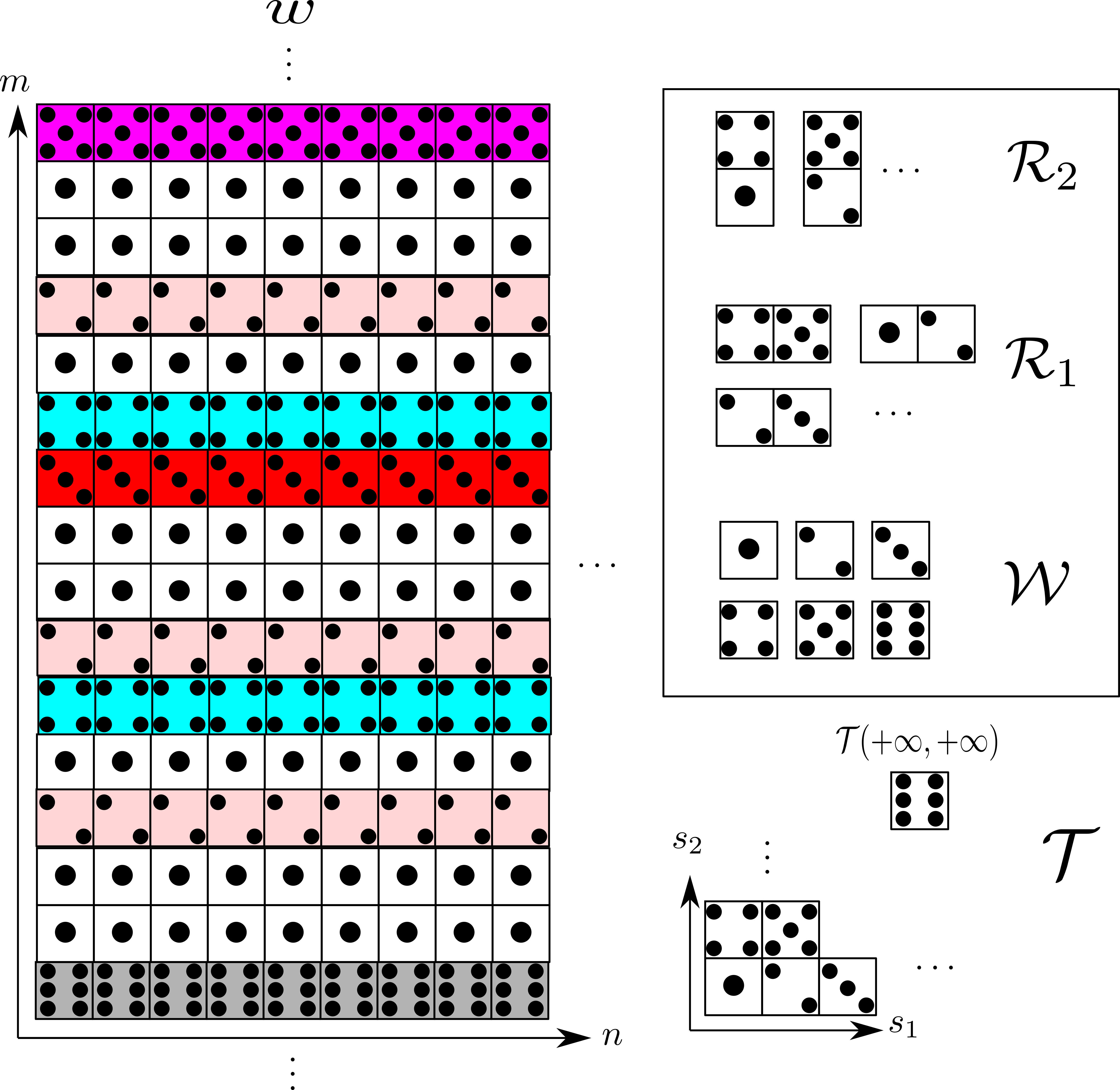}
    \caption{A portion of the ``decoration'' function $w(n,m) \coloneqq \cT( \nu_{3,5}(m) )$, where $\cT$ is the ${\mathcal R}$-domino function depicted on the right-hand side for the indicated domino set ${\mathcal R} = ({\mathcal W}, {\mathcal R}_1, {\mathcal R}_2)$, and we have arbitrarily assigned the $6$-pip to $\cT(+\infty,+\infty)$.  The triple $(F_1,F_2,w)$, where $F_1, F_2$ were defined in Figures \ref{fig:p3}, \ref{fig:p5} (where $F_1$ and $F_2$ are now extended to the board of width $N=9\times 25$), will then be a $\cS^\cR$-Sudoku solution.  Here we have used the parameter choices $(p_1,p_2,N)=(3,5,9\times 25)$ for sake of illustration, despite the fact that these choices do not obey the hypotheses of Proposition \ref{dom-solv}.  The coloring scheme for $w$ is a superposition of those in Figures \ref{fig:p3}, \ref{fig:p5}, with light purple (combination of pink and cyan) corresponding to the case $\nu_{3,5}(m)=(1,1)$; this scheme resembles the outcome of the children's game ``Fizz buzz''.} 
    \label{fig:domino}
\end{figure}

\begin{proof}  First suppose that the $\cR$-domino problem has a solution $\cT \colon \N^2 \to \cW$.  Then we can define the function $F = (F_1,F_2,w)$ by the formulae
\begin{align*}
    F_1(n,m) &\coloneqq f_{p_1}(m) \\
    F_2(n,m) &\coloneqq f_{p_2}(m) \\
    w(n,m) &\coloneqq \cT( \nu_{p_1,p_2}(m) )
\end{align*}
where we arbitrarily assign a value in $\cW$ to $\cT(+\infty,+\infty)$; see Figure \ref{fig:domino}.  Clearly $F_1, F_2$ have non-constant columns.  We claim that $F$ is a $\cS^\cR$-solution, thus we need $$n \mapsto F(n,jn+i)$$ to lie in $\cS^\cR$ for every $j,i \in \Z$.  In the degenerate case $j=i=0$ this follows from Definition \ref{sudoku-construction}(iii) by taking the constant (but still non-degenerate modulo $p_1,p_2$) affine form $n \mapsto 1$ (so that $t_1=t_2=0$) and the domino function on the $1 \times 1$ rectangle $[(0,0),(0,0)]$ is given by $(0,0) \mapsto \cT(+\infty,+\infty)$, which is vacuously a domino function since a $1 \times 1$ rectangle contains no domino tiles.
    
Now suppose that $(j,i) \neq (0,0)$.  Then for $l=1,2$ we may write $(j,i) = p_l^{d_l} (j_l,i_l)$, where $d_l$ is the natural number $d_l \coloneqq \min(\nu_{p_l}(j), \nu_{p_l}(i))$ and $j_l, i_l$ are integers  that are not both divisible by $p_l$.  From \eqref{fp-mult} we then have
$$ F(n,jn+i) = ( f_{p_1}(j_1 n + i_1), f_{p_2}(j_2 n + i_2), \cT( (\nu_{p_1}(j_1n + i_1), \nu_{p_2}(j_2 n + i_2)) + (d_1,d_2) ).$$
By the Chinese remainder theorem, we can find integers $a,b$ such that $$a = j_l \pmod{p_l^2},\quad b = i_l \pmod{p_l^2}$$ for $l=1,2$; in particular, the affine form $n \mapsto an+b$ is non-degenerate modulo both $p_1$ and $p_2$.  From \eqref{fp-period} we then have
$$ F(n,jn+i) = ( f_{p_1}(a n + b), f_{p_2}(a n + b), \cT( \nu_{p_1,p_2}(a n + b) + (d_1,d_2) )$$ 
whenever $\nu_{p_1,p_2}(an+b) \in [(0,0),(1,1)]$.  In particular, this holds whenever
$\nu_{p_1,p_2}(an+b) \in [(0,0),(t_1,t_2)]$, where $t_1,t_2 \in \{0,1\}$ are defined as in Definition \ref{sudoku-construction}(iii). 
Since $\cT$ is a $\cR$-domino function on $\N^2$, the translated function $s \mapsto \cT( s + (d_1,d_2) )$ is a $\cR$-domino function on $[(0,0), (t_1,t_2)]$, and hence $n \mapsto F(n,jn+i)$ lies in $\cS^\cR$ as required.   

Conversely, suppose that $F = (F_1,F_2,w)$ is a $\cS^\cR$-Sudoku solution with $F_1,F_2$ having non-constant columns.  By Lemma \ref{dom-equiv}, it suffices to show that the $\cR$-domino problem is solvable on $[(0,0), r]$ for any $r \geq (0,0)$.

Fix $r = (r_1,r_2) \geq (0,0)$.  By Remark \ref{alt}, for $l=1,2$ we may find $c_l \in (\Z/p\Z)^\times$ and $D_l, E_l \in \Z_{p_l}$ such that
$$ F_l(n,m) = c_l f_{p_l}( m - D_l n - E_l  )$$
whenever $(n,m) \in \mathbb{B}$ with $m \neq D_l n + E_l$.  By the Chinese remainder theorem, we may find integers $c, D, E$ such that $$\pi_{p_l}(c) = c_l,\quad D = D_l \mod{p_l^{r_l}},\quad E = E_l \mod{p_l^{r_l}}$$ for $l=1,2$. By \eqref{fp-mult}, \eqref{fp-period} we then have
$$ F_l(n,m) = f_{p_l}( c (m - D n - E)  )$$
for $l=1,2$ whenever $(n,m) \in \mathbb{B}$ with $\nu_{p_l}(m - Dn - E) \leq r_l$.  We may now apply a shear transformation to replace $F(n,m)$ by $F(n,m+Dn+E)$ which, by Proposition \ref{sudoku-affine}(ii), does not affect the property of $F$ being a $\cS^\cR$-Sudoku solution with $F_1,F_2$ having non-constant columns. Thus, this allows us to eliminate $D, E$, so that 
\begin{equation}\label{flnm}
F_l(n,m) = f_{p_l}( c m  )
\end{equation}
whenever $l=1,2$ and $(n,m) \in \mathbb{B}$ with $\nu_{p_l}(m) \leq r_l$.  

Now we focus on the function $w$.  For any non-vertical line $\ell_{j,i}$, we see from Definition \ref{sudoku-def} and Definition \ref{sudoku-construction}(iii) that there exist an affine form $$n \mapsto a_{j,i} n + b_{j,i},$$ non-degenerate modulo $p_l$ for $l=1,2$, as well as a
$\cR$-domino function $$\cT_{j,i} \colon [(0,0),(t_{j,i,1},t_{j,i,2})] \to \cW$$
where, for each $l=1,2$, $t_{j,i,l}$ is equal to $1$ when $\nu_{p_l}(a_{j,i}) = 0$ and equal to $0$ otherwise, such that
\begin{equation}\label{flji}
F_l(n, jn+i) = f_{p_l}( a_{j,i} n + b_{j,i} )
\end{equation}
for $l=1,2$ and $n \in \{0,\dots,N-1\}$ with $\nu_{p_l}(a_{j,i} n + b_{j,i}) \leq t_{j,i,l}$, and
\begin{equation}\label{wji}
 w(n, jn+i) = \cT_{j,i}( \nu_{p_1,p_2}( a_{j,i} n + b_{j,i} ) )
 \end{equation}
whenever $n \in \{0,\dots,N-1\}$ with $\nu_{p_1,p_2}(a_{j,i} n + b_{j,i})) \leq (t_{j,i,1},t_{j,i,2})$,

We analyze these properties for various choices of non-vertical line $\ell_{j,i}$.  First consider the case of a row $\ell_{0,i}$ with $\nu_{p_1,p_2}(i) \leq r$.  From \eqref{flnm}, \eqref{flji} we have for each $l=1,2$ that
$$
F_l(n, i) =f_{p_l}(ci) = \pi_{p_l}( a_{0,i} n + b_{0,i})$$
for  all $n \in \{0,\dots,N-1\}$ outside of at most two cosets of $p_l \Z$.  From Lemma \ref{basic-affine}(iii), we conclude that $\pi_{p_l}(a_{0,i})=0$ and $\pi_{p_l}(b_{0,i}) = f_{p_l}(ci)$. From \eqref{wji} we then see that
$$ w(n,i) = \cT_{0,i}(0,0)$$
for all $n \in \{0,\dots,N-1\}$.  In particular, $w(n,i)$ is independent of $n$ whenever $\nu_{p_1,p_2}(i) \leq r$.

Next, consider a line $\ell_{j,i}$ with $\nu_{p_1,p_2}(i) \leq r$ and $\nu_{p_l}(j) > \nu_{p_l}(i)$ for $l=1,2$.  From \eqref{fp-period}, \eqref{fp-mult} we have 
$$f_{p_l}(c(jn+i)) = f_{p_l}(ci)\; \hbox{ and }\;\nu_{p_l}(c(jn+i)) = \nu_{p_l}(ci)$$ 
for $l=1,2$, so by arguing as before we have
$$ w(n,jn+i) = \cT_{j,i}(0,0)$$
for all $n \in \{0,\dots,N-1\}$; in particular, 
$$w(0,i) = w(1,j+i).$$
Therefore, since $w(n,m)$ is independent of $n$ when $\nu_{p_1,p_2}(m) \leq r$, we have
\begin{equation}
    w(n,i) = w(0,i)=w(1,j+i)=w(n',j+i) 
\end{equation}
whenever $n,n' \in \{0,\dots,N-1\}$ and $\nu_{p_l}(i) \leq r_l$ and $\nu_{p_l}(j) > \nu_{p_l}(i)$ for $l=1,2$.
Note that $\nu_{p_l}(j) > \nu_{p_l}(i)$ if and only if $f_{p_l}(i) = f_{p_l}(j+i)$; thus, making the change of variables $m=i$, $m'=j+i$, we conclude that
$$ w(n,m) = w(n',m')$$
whenever $(n,m), (n',m') \in \mathbb{B}$ are such that 
$$\nu_{p_1,p_2}(m) = \nu_{p_1,p_2}(m') \leq r \; \hbox{ and }\; f_{p_1,p_2}(m) = f_{p_1,p_2}(m').$$ 
In other words, the value of $w(n,m)$ depends only on $\nu_{p_1,p_2}(m)$ and $f_{p_1,p_2}(m)$, so long as $(\nu_{p_1,p_2}(m)) \leq r$, thus one has
\begin{equation}\label{tw}
    w(n,m) = \tilde w( \nu_{p_1,p_2}(m), f_{p_1,p_2}(m) )
\end{equation}
whenever $(n,m) \in \mathbb{B}$ is such that $(\nu_{p_1,p_2}(m)) \leq r$, for some function $$\tilde w \colon [(0,0),r] \times (\Z/p_1\Z)^\times \times (\Z/p_2\Z)^\times \to \cW.$$

In fact, we can remove the dependence on $f_{p_1,p_2}(m)$ as follows.  Let $0 \leq s_1 \leq r_1$ and $0 \leq s_2 \leq r_2$, and consider the line $\ell_{j,i}$ with $i \coloneqq 0$ and $j \coloneqq p_1^{s_1} p_2^{s_2}$.  From \eqref{flnm}, \eqref{fp-mult}, \eqref{fp-period} we have for $l=1,2$ that
$$ F_l(n,jn) = \pi_{p_l}(c n) $$
for all $n \in \{0,\dots,N-1\}$ not divisible by $p_l$, while from \eqref{flji} we have
$$ F_l(n, jn) = \pi_{p_l}( a_{j,0} n + b_{j,0} )$$
for all $n \in \{0,\dots,N-1\}$ outside of a coset of $p_l\Z$.  We conclude from Lemma \ref{basic-affine}(iii) that the affine forms $a_{j,0} n + b_{j,0}$ and $cn$ are identical modulo $p_l$ for $l=1,2$.  From \eqref{wji}, we conclude that
\begin{equation}\label{njn}
 w(n, jn) = \cT_{j,0}( 0, 0 )
 \end{equation}
whenever $n \in \{0,\dots,N-1\}$ is coprime to $p_1,p_2$.  For such $n$, we have $\nu_{p_l}(jn) = s_l$, while $(f_{p_1}(jn), f_{p_2}(jn))$ traverses every pair in $(\Z/p_1\Z)^\times \times (\Z/p_2\Z)^\times$. Comparing this with \eqref{tw} we conclude that the function $\tilde w$ is constant in the second two variables, thus we can in fact write
\begin{equation}\label{njn-2}
w(n,m) = \cT( \nu_{p_1}(m), \nu_{p_2}(m) )
\end{equation}
for some function $\cT \colon [(0,0),r] \to \cW$ and all $(n,m) \in \mathbb{B}$ with $(\nu_{p_1}(m), \nu_{p_2}(m)) \leq r$.

Now let $[s,s+e_1]$ be a horizontal domino tile in $[(0,0), r]$ for some $s = (s_1,s_2)$. Thus $0 \leq s_1 \leq r_1-1$ and $0 \leq s_2 \leq r_2$.  We again consider the line $\ell_{j,i}$ with $i \coloneqq 0$ and $j \coloneqq p_1^{s_1} p_2^{s_2}$.  We have already seen that 
$a_{j,0} n + b_{j,0}$ and $cn$ are identical modulo $p_1$ and modulo $p_2$, so in particular $\nu_{p_1}(a_{j,0})=0$ and $\nu_{p_1}(b_{j,0}) \geq 1$; among other things, this forces $t_{j,i,1}=1$.  From \eqref{flnm}, \eqref{fp-mult}, \eqref{fp-period} we have
$$ F_1(p_1 n,jp_1 n) = \pi_{p_1}(c n) $$
when $n \in \{1,\dots,N/p_1\}$ is coprime to $p_1$, while  from \eqref{flji} we have
$$ F_1(p_1n, jp_1 n) = \pi_{p_1}\left( a_{j,0} n + \frac{b_{j,0}}{p_1} \right)$$
when $n \in \{1,\dots,N/p_1\}$ outside of a coset of $p_l\Z$. We conclude from Lemma \ref{basic-affine}(iii) that the affine forms 
$$n\mapsto a_{j,0} n + \frac{b_{j,0}}{p_1}, \quad n\mapsto cn$$ 
are identical modulo $p_1$.  Recall also that $a_{j,0} n + b_{j,0}$ and $cn$ are identical modulo $p_2$.  We conclude that if $n$ is coprime to both $p_1$ and $p_2$, then $\nu_{p_1}( a_{j,0} p_1 n + b_{j,0} ) = 1$ and $\nu_{p_2}( a_{j,0} p_1 n + b_{j,0} ) = 0$. Applying \eqref{wji}, we therefore have
$$
w(p_1 n, jp_1 n) = \cT_{j,0}( 1, 0 )
$$
for such $n$.  From this and \eqref{njn}, \eqref{njn-2} we have
 $$ \cT_{j,0}( 0, 0 ) = \cT(s_1,s_2); \quad  \cT_{j,0}( 1, 0 ) = \cT(s_1+1,s_2).$$
 Since $\cT_{j,0}$ is a $\cR$-domino, we conclude that
 $$ (\cT(s_1,s_2), \cT(s_1+1,s_2)) \in \cR_1$$
 whenever $[s,s+e_1]$ is a horizontal domino tile in $[(0,0), r]$.  A similar argument gives
 $$ (\cT(s_1,s_2), \cT(s_1,s_2+1)) \in \cR_2$$
 whenever $[s,s+e_2]$ is a vertical domino tile in $[(0,0), r]$.  Thus $\cT$ is an $\cR$-domino function on $[(0,0),r]$, and the claim follows.
\end{proof}

In view of the above proposition, we see that to prove Theorem \ref{domino-to-sudoku}, it will suffice to provide an initial condition $\cC = \cC_q$ with the property that a  $\cS^\cR$-Sudoku solution $F = (F_1,F_2,w)$ obeys $\cC$ if and only if $F_1, F_2$ have non-constant columns.  This will be accomplished by setting
$$ q \coloneqq p_1 p_2$$
and defining $\cC \subset (\Z/q\Z) \times (\Z/p_1\Z)^\times \times (\Z/p_2\Z)^\times \times \cW$ to be the set of quadruplets
$$ (a, b_1, b_2, w) \in (\Z/q\Z) \times (\Z/p_1\Z)^\times \times (\Z/p_2\Z)^\times \times \cW$$
such that one of the following statements holds:
\begin{itemize}
    \item $a$ is not coprime to $q$.
    \item $b_1 = a \pmod{p_1}$ and $b_2 = a \pmod{p_2}$.
\end{itemize}

Suppose that $F = (F_1,F_2,w)$ is a $\cS^\cR$-Sudoku solution that obeys the initial condition $\cC$.  Then by Definition \ref{sudoku-def}, for any $n \in \{0,\dots,N-1\}$ there exists a permutation $\sigma_n \colon \Z/q\Z \to\Z/q\Z$ such that for any $m \in \Z$ with $\sigma_n(\pi_q(m))$ coprime to $q$, we have 
\begin{equation}\label{flmns}
F_l(n,m) = \pi_{p_l}(\sigma_n(\pi_q(m)))
\end{equation}
for $l=1,2$.  In particular, the maps $m \mapsto F_l(n,m)$ are not constant, thus $F_1, F_2$ have non-constant columns.

Conversely, suppose that $F_1, F_2$ have non-constant columns.  Applying Theorem \ref{initial-prop}, we see that for $l=1,2$ there exists a non-degenerate affine form 
$$(n,m) \mapsto A^{(0)}_ln+B^{(0)}_lm+C^{(0)}_l$$ 
with $\nu_{p_l}(B^{(0)}_l)=0$ such that
$$
F_l(n,m) = \pi_{p_l}(A^{(0)}_ln+B^{(0)}_lm+C^{(0)}_l)
$$
whenever $(n,m) \in \mathbb{B}$ is such that $\nu_{p_l}(A^{(0)}_ln+B^{(0)}_lm+C^{(0)}_l)=0$. By the Chinese remainder theorem we can find integers $B, D, E$ with $B$ coprime to $q$ such that the affine forms $$(n,m) \mapsto B(m - Dn - E),\quad  (n,m) \mapsto A^{(0)}_l n + B^{(0)}_l m + C^{(0)}_l$$  are identical modulo $p_l$ for $l=1,2$.  If we then define $\sigma_n \colon \Z/q\Z \to \Z/q\Z$ by $$\sigma_n(a) \coloneqq B(a-Dn-E),$$ then $\sigma_n$ is a permutation and
\eqref{flmns} holds whenever $\sigma_n(\pi_q(m))$ is coprime to $q$.  Thus the initial condition $\cC$ holds.  This completes the desired equivalence of the non-constant columns condition and the initial condition $\cC$, and Theorem \ref{domino-to-sudoku} follows. 

See Figure \ref{fig:sec4-5} for a high level overview of the proof of Theorem \ref{domino-to-sudoku}.

\begin{figure}
    \centering
    \includegraphics[width = .8\textwidth]{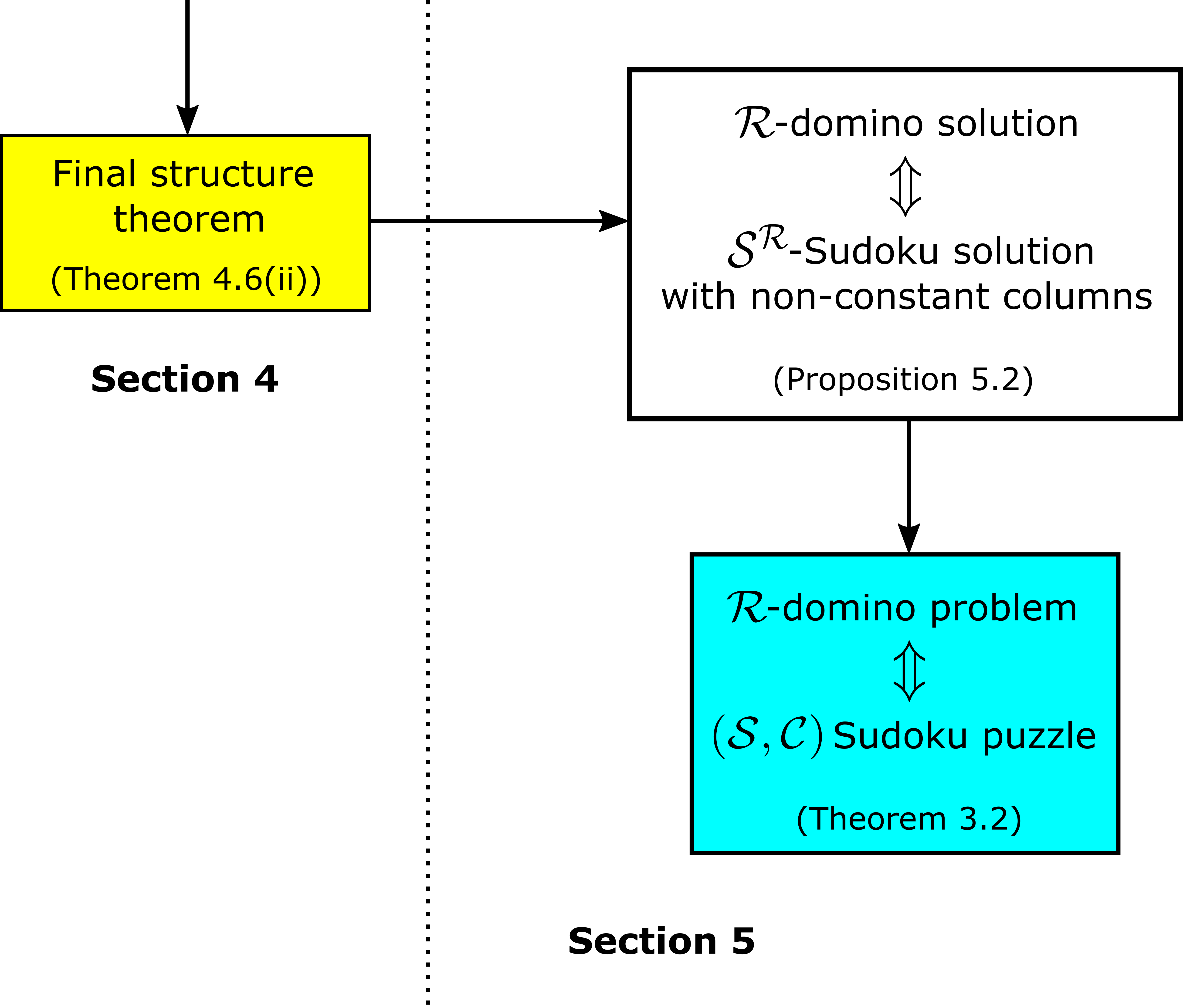}
    \caption{Combined with Figure \ref{fig:sec4}, this gives a high level illustration of the proof of Theorem \ref{domino-to-sudoku} in Sections \ref{sec:4} and \ref{sec:5}.}
    \label{fig:sec4-5}
\end{figure}

\section{Encoding Sudoku puzzles as monotiling problems}\label{sec:6}

In this section we establish Theorem \ref{sudoku-to-monotile}.  Our arguments here will be similar to those in \cite[Sections 6,7]{GT22}, though our task is made easier here by the fact that we permit ourselves to consider monotiling problems in periodic subsets $\Z^2 \times E$ of the ambient group $\Z^2 \times G$, as opposed to the full group, which (as in \cite{GT21}) allows for a simpler and more expressive ``tiling language''.

It will be convenient to adapt the definition of an \emph{expressible property} and a \emph{weakly expressible property} from \cite[Definitions 4.2, 4.13]{GT22}.

\begin{definition}[Expressible properties]\label{expressible-def}  Let $G = (G,+)$ and $H=(H,+)$ be Abelian groups with $G$ finitely generated and $H$ finite.  
\begin{itemize}
    \item[(i)] A \emph{$(G,H)$-property} is a property $P$ of a function $\alpha \colon G \to H$ (or equivalently, a subset of $H^G$).  The property $P$ is \emph{satisfiable} if there is at least one function $\alpha \colon G \to H$ that obeys $P$.
    \item[(ii)]  A $(G,H)$-property $P$ is \emph{expressible in the language of functional equations with subsets}, or \emph{expressible} for short, if there exist a natural number $M$, a natural number $J_i$ and a subset $E'_i$ of $H$ for all $i=1,\dots,M$, and shifts $h_{i,j} \in G$, and subsets $E_{i,j}$ of $H$ for all $i=1,\dots,M$ and $j=1,\dots,J_i$, such that a function $\alpha \colon G \to H$ obeys the system of functional equations
    \begin{equation}\label{aeij}
    \biguplus_{j=1}^{J_i}(\alpha(x+h_{i,j})+E_{i,j}) = E'_i \text{ for all } i=1,\dots,M \text{ and } x\in G
    \end{equation}
    if and only if it obeys the property $P$. 
    \item[(iii)]  If $(H_u)_{u \in {\mathcal U}}$ is a finite collection of finite Abelian groups $H_u = (H_u,+)$, a \emph{$(G,(H_u)_{u \in {\mathcal U}})$-property} is a property $P$ of a tuple of functions $$\alpha_u \colon G \to H_u,\quad u\in\mathcal U,$$ or equivalently, a single function $$\alpha \colon G \to \prod_{u \in {\mathcal U}} H_u.$$   The $(G,(H_u)_{u \in {\mathcal U}})$-property is \emph{expressible} (in the above language) if the corresponding $(G,\prod_{u \in {\mathcal U}} H_u)$-property is expressible.  Similarly, a $(G,(H_u)_{u \in {\mathcal U}})$-property $P$ is \emph{satisfiable} if there exists a tuple of functions $$\alpha_u \colon G \to H_u, \quad u \in {\mathcal U}$$ that obey $P$.
    \item[(iv)] If $(H_u)_{u \in {\mathcal U} \uplus {\mathcal U}^*}$ is a finite collection of finite Abelian groups $H_u = (H_u,+)$ indexed by the disjoint union of two index sets ${\mathcal U}, {\mathcal U}^*$, and $P^*$ is a $(G, (H_u)_{u \in {\mathcal U} \uplus {\mathcal U}^*})$-property, we define the \emph{existential quantification} of $P^*$ to be the $(G, (H_u)_{u \in {\mathcal U}})$-property $P$ defined by requiring a tuple $(\alpha_u)_{u \in {\mathcal U}}$ of functions $$\alpha_u \colon G \to H_u,\quad u\in\mathcal U$$ to obey $P$ if and only if there exists an extension $(\alpha_u)_{u \in {\mathcal U} \uplus {\mathcal U}^*}$ of this tuple that obeys $P^*$.  If a property $P$ arises as the existential quantification of some expressible property $P^*$, we say that $P$ is \emph{weakly expressible} (in the above language). For instance, any property which is expressible is also weakly expressible.
\end{itemize}
\end{definition}

See \cite[Section 4]{GT22} for several examples of expressible and weakly expressible properties.

\begin{remark}\label{rem:persub} In \cite{GT22}, the only subset $E'_i$ of $H$ one was permitted to use in \eqref{aeij} was the whole space $E'_i = H$, and so the notions of expressiveness and weak expressiveness were more restricted than the ones used here; hence our addition of the modifier ``with subsets'' in the current notation.  This reflects the fact that in \cite{GT22} we restricted attention to translational tilings of an entire Abelian group, as opposed to periodic subsets of that group.  It seems plausible that this is merely a technical restriction, and some version of the Sudoku puzzles under consideration could be encoded without using subsets, in which case the set $E$ in Theorem \ref{thm:main} could be taken to be all of $G_0$ (and similarly the set $E$ in Corollary \ref{cor:mainZd} could be taken to be all of $\Z^d$).  Unfortunately, our current library of expressive properties (without subsets) only allows us to express initial conditions involving a period $q$ which is a power of two, whereas here we need $q$ to be divisible by two large primes, so either some modification of the Sudoku puzzle or some enlargement of the library would be needed.
\end{remark}

  The relevance of these properties to tiling problems is given by the following proposition.

\begin{proposition}[Connection between expressible properties and monotiling problems] Let $G = (G,+)$ and $H=(H,+)$ be Abelian groups with $G$ being finitely generated and $H$ finite.  
\begin{itemize}
    \item[(i)]  If one is given an expressible $(G,H)$-property $P$, then one can construct (in finite time) a finite Abelian group $H'$, a subset $E'$ of $H'$, and a finite subset $F$ of $G \times H'$, with the property that $F$ tiles $E'$ if and only if $P$ is satisfiable.
    \item[(ii)]  If one is given an expressible $(G,(H_u)_{u \in {\mathcal U}})$-property $P$, then one can construct (in finite time) a finite Abelian group $H'$, a subset $E'$ of $H'$, and a finite subset $F$ of $G \times H'$, with the property that $F$ tiles $E'$ if and only if $P$ is satisfiable.
    \item[(iii)]  If one is given a weakly expressible $(G,(H_u)_{u \in {\mathcal U}})$-property $P$, then one can construct (in finite time) a finite Abelian group $H'$, a subset $E'$ of $H'$, and a finite subset $F$ of $G \times H'$, with the property that $F$ tiles $E'$ if and only if $P$ is satisfiable.
\end{itemize}
\end{proposition}

\begin{proof}  We begin with (i), which is a variant of \cite[Theorem 4.1]{GT22}.  By Definition \ref{expressible-def}(ii), the property $P$ is equivalent to the system \eqref{aeij} for some suitable data $M, J_i, E'_i, h_{i,j}, E_{i,j}$.  One can then check (as in the proof of \cite[Theorem 4.1]{GT22}) that if $\alpha \colon G \to H$ obeys $P$, then the graph
\begin{equation}\label{graph}
A \coloneqq \{ (x, \alpha(x)): x \in G \} \subset G \times H
\end{equation}
obeys the system of tiling equations
\begin{equation}\label{agh}
 A \oplus (\{0\} \times H) = G \times H
 \end{equation}
and
\begin{equation}\label{agh-2}
A \oplus \biguplus_{j=1}^{J_i} \{-h_{i,j}\} \times E_{i,j} = G \times E'_i \text{ for all } i=1,\dots,M.
\end{equation}
Conversely, if $A \subset G \times H$ is a set that obeys the system \eqref{agh}, \eqref{agh-2}, then $A$ is a graph \eqref{graph} of some function $\alpha \colon G \to H$ (thanks to \eqref{agh}), and $\alpha$ will obey $P$ (thanks to \eqref{agh-2} and \eqref{aeij}).  To conclude (i), one can then invoke \cite[Theorem 1.15]{GT21} to merge together (via a finite time construction) this system of tiling equations into a single tiling equation for a subset $G \times E'$ of some group $G \times H'$, in such a way that the latter equation is solvable if and only if the former system is solvable.

The claim (ii) is immediate from (i) after re-interpreting $P$ as a $(G,\prod_{u \in {\mathcal U}} H_u)$-property.  Finally, to derive (iii) from (ii), observe that if $P$ is the existential quantification of some property $P^*$, then $P$ is satisfiable if and only if $P^*$ is.
\end{proof}

In view of this proposition, we now see that Theorem \ref{sudoku-to-monotile} is reduced to the following claim.

\begin{theorem}[Encoding Sudoku puzzles as weakly expressible properties]\label{sudoku-to-expressible}  Suppose one is given a Sudoku rule $\cS$ and an initial condition $\cC$.  Then one can generate (in finite time) a finite Abelian group $H_1$, and a weakly expressible $(\Z^2 \times \Z/2\Z, H_1)$-property $P$, such that the $(\cS, \cC)$ Sudoku puzzle is solvable if and only if $P$ is satisfiable.  Furthermore, the weak expressibility is constructive in the sense that the system of functional equations that witness the weak expressibility of $P$ can also be generated in finite time from $\cS$ and $\cC$.
\end{theorem}

The advantage of working with weakly expressive properties, as opposed to tiling equations, is that the class of such properties obey a number of useful closure properties.  We recall a definition from \cite[Definition 4.15, 4.18]{GT22}:

\begin{definition}[Lift]\label{lift}
Given a finitely generated Abelian group $G$, a tuple of finite Abelian groups $(H_u)_{u \in \mathcal{U}}$ indexed by a finite set $\mathcal{U}$, a subset $\mathcal{U}_1$ of $\mathcal{U}$,  and a $(G, (H_u)_{u \in \mathcal{U}_1})$-property $P_1$, we define the \emph{lift} of $P_1$ to $(G, (H_u)_{u \in \mathcal{U}})$ to be the $(G, (H_u)_{u \in \mathcal{U}})$-property $P$, defined by requiring a $(G, (H_u)_{u \in \mathcal{U}})$-function $(\alpha_u)_{u \in \mathcal{U}}$ to obey $P$ if and only if the $(G, (H_u)_{u \in \mathcal{U}_1})$-function $(\alpha_u)_{u \in \mathcal{U}_1}$ obeys $P_1$.  
\end{definition}

Informally, applying a lift to a property adds some additional ``dummy'' functions $\alpha_u$, $u \in {\mathcal U} \backslash {\mathcal U}_1$.  We refer to \cite[Section 4]{GT22} for examples and discussion of this operation.

We then have

\begin{lemma}[Closure properties of expressibility and weak expressibility]\label{closure}\ 
\begin{itemize}
    \item[(i)]  Any lift of an expressible (resp. weakly expressible) property is also expressible (resp. weakly expressible).
    \item[(ii)]  The conjunction $P \wedge P'$ of two expressible (resp. weakly expressible) $(G, (H_u)_{u \in \mathcal{U}})$-properties is also expressible (resp. weakly expressible).
    \item[(iii)] Any existential quantification of a weakly expressible property is weakly expressible.
\end{itemize}
Furthermore, if the (weak) expressibility of the properties in the hypotheses is constructive, then so is the (weak) expressibility of the properties in the conclusion.
\end{lemma}

\begin{proof} This is\footnote{This lemma also included a closure property involving a ``pullback'' operation which we will not use here.} \cite[Lemma 4.20]{GT22}, generalized to the setting in which subsets $E'_i$ are permitted in the functional equations \eqref{aeij}; it is a routine matter to verify that the proof of that lemma extends to this setting with the obvious modifications, and that the arguments preserve the constructive nature of the expressibility.
\end{proof}

\begin{example}  Let $P$ be a weakly expressible $(G,H)$-property, and $P'$ an expressible $(G,H')$ property.  Then, using $(\alpha,\alpha')$ to refer to a tuple of functions $\alpha \colon G \to H$ and $\alpha' \colon G \to H'$, then Lemma \ref{closure} implies that $(G,(H,H'))$-property of $\alpha$ obeying $P$ is weakly expressible thanks to Lemma \ref{closure}(i), and similarly the $(G,(H,H'))$-property of $\alpha'$ obeying $P'$ is expressible.  By Lemma \ref{closure}(ii), we conclude that the $(G,(H,H'))$-property of $\alpha$ obeying $P$ \emph{and} $\alpha'$ obeying $P'$ is also weakly expressible.  This type of combination of Lemma \ref{closure}(i) and Lemma \ref{closure}(ii) will be frequently used in the arguments below.
\end{example}

\subsection{A library of weakly expressible properties}

Similarly to \cite[Sections 5, 6, 7]{GT22}, we now build up a library of useful expressible or weakly expressible properties.

\begin{lemma}[Expressing periodicity]\label{express-period}  Let $G$ be a finitely generated Abelian group, let $H$ be a finite Abelian group, and let $G'$ be a subgroup of $G$.  Then the $(G,H)$-property that a $(G,H)$-function $\alpha$ is $G'$-periodic in the sense that $\alpha(x+h) = \alpha(x)$ for all $x \in G$ and $h \in G'$, is expressible (in a constructive fashion).
\end{lemma}

\begin{proof}  See \cite[Corollary 5.4]{GT22}.  Alternatively, for $h_1,\dots,h_k$ a set of generators for $G'$, the periodicity can be expressed as a system of functional equations
$$ (\alpha(x+h_i) + \{0\}) \uplus (\alpha (x) + (H \backslash \{0\})) = H$$
for all $x\in G$ and $i=1,\dots,k$, thus making the expressibility immediate.
\end{proof}

\begin{lemma}[Expressing linear constraints]\label{express-linear} (cf. \cite[Corollary 5.5]{GT22}, \cite[Section 6]{GT21}) Let $G$ be a finitely generated Abelian group, let $\Z/L\Z$ be a cyclic group, and let $c_1,\dots,c_U \in \Z/L\Z$ be coefficients.  Then the $(G,(\Z/L\Z)_{u=1,\dots,U})$-property of a tuple $\alpha_1,\dots,\alpha_U \colon G \to \Z/L\Z$ of functions obeying the linear relation
\begin{equation}\label{ca}
c_1 \alpha_1(x) + \dots + c_U \alpha_U(x) = 0
\end{equation}
for all $x \in G$, is expressible (in a constructive fashion).
\end{lemma}

\begin{proof}  The condition \eqref{ca} is equivalent to the condition
$$ (\alpha_1(x),\dots,\alpha_U(x)) + E = E$$
where $E \subset (\Z/L\Z)^U$ is the subspace
$$ E \coloneqq \{ (a_1,\dots,a_U) \in (\Z/L\Z)^U : c_1 a_1 + \dots + c_U a_U = 0 \}.$$
The expressibility then follows.
\end{proof}

\begin{definition}[Boolean function] (cf. \cite[Definition 6.1]{GT22}) Let $G$ be a finitely generated Abelian group, let $e$ be an element of $G$ of order $2$, and let $\Z/L\Z$ be a cyclic group for some $L > 2$. A function $\alpha \colon G \to \Z/L\Z$ is \emph{$e$-boolean} if it takes values in $\{-1,+1\}$ (viewing $-1, +1$ as elements of $\Z/L\Z$, and furthermore obeys the alternating property
\begin{equation}\label{alternating}
\alpha(x + e) = - \alpha(x)
\end{equation}
for all $x \in G$.
\end{definition}

\begin{lemma}[$e$-boolean functions are expressible]\label{bool-function}  If $G, e, L$ are as in the above definition, then the $(G, \Z/L\Z)$-property of being an $e$-boolean function is expressible (in a constructive fashion).
\end{lemma}

\begin{proof}
    Observe that a function $\alpha \colon G \to \Z/L\Z$ obeys the functional equation
    $$ (\alpha(x+e) + \{0\}) \uplus (\alpha(x) + \{0\}) = \{-1,+1\}$$
     if and only if $\alpha$ is $e$-boolean.  The claim follows.
\end{proof}

\begin{lemma}[Boolean constraints are expressible]\label{boolean-express} (cf. \cite[Proposition 6.6]{GT22} and \cite[Section 6]{GT21})  Let $G$ be a finitely generated Abelian group, let $e$ be an element of $G$ of order $2$, let $U \geq 1$, let $\Omega \subset \{-1,1\}^U$ obey the symmetry condition $-\Omega = \Omega$, and let $\Z/L\Z$ be a cyclic group for some $L > 2U+4$.  Then the $(G, (\Z/L\Z)_{u=1,\dots,U})$ property of a tuple of functions $\alpha_u \colon G \to \Z/L\Z$, $u=1,\dots,U$ being $e$-boolean and obeying the constraint
$$ (\alpha_1(x),\dots,\alpha_U(x)) \in \Omega \text{ for all } x \in G$$
is weakly expressible.
\end{lemma}

\begin{proof}
By increasing $U$ by one or two if necessary (and relaxing $L > 2U+4$ to $L > 2U$) using Lemma \ref{closure}(iii), we may assume without loss of generality that $U$ is odd with $U \geq 3$.  The symmetric set $\Omega$ can be expressed as the intersection of a finite number of symmetric sets of the form
\begin{equation}\label{single}
 \{-1,+1\}^U \backslash \{ (\epsilon_1,\dots,\epsilon_U), (-\epsilon_1,\dots,-\epsilon_U) \}
\end{equation}
for some $\epsilon_1,\dots,\epsilon_U \in \{-1,+1\}$.  By Lemma \ref{closure}(ii), it thus suffices to verify the claim for $\Omega$ of the form \eqref{single}.  

We introduce some auxiliary functions $\beta_1,\dots,\beta_{U-2} \colon G \to \Z/L\Z$, and let $P^*_\Omega$ be the $(G, (\Z/L\Z)_{u=1,\dots,2U-2})$-property that a tuple $(\alpha_1,\dots,\alpha_U,\beta_1,\dots,\beta_{U-2})$ of functions from $G$ to $\Z/L\Z$ are $e$-boolean and obey the linear constraint
$$ \epsilon_1 \alpha_1(x)+\dots+ \epsilon_U \alpha_U(x) = \beta_1(x) + \dots + \beta_{U-2}(x) $$
    for all $x \in G$.

From Lemma \ref{express-linear}, Lemma \ref{bool-function} and Lemma \ref{closure}(ii), the property $P^*_\Omega$ is weakly expressible.  By Lemma \ref{closure}(iii), it thus suffices to show that $P_\Omega$ is the existential quantification of $P^*_\Omega$.

We first show that any tuple $(\alpha_1,\dots,\alpha_U)$ obeying $P_\Omega$ can be extended to a tuple $(\alpha_1,\dots,\alpha_U,\beta_1,\dots,\beta_{U-2})$ obeying $P^*_\Omega$.  By the property $P_\Omega$, for any $x \in G$, the expression $\epsilon_1 \alpha_1(x)+\dots+ \epsilon_U \alpha_U(x)$ will have magnitude at most $U-2$, and is odd since $U$ is odd.  Since $U-2$ is also odd, we conclude that this expression can be written in the form $\beta_1(x) + \dots + \beta_{U-2}(x)$ for some $\beta_1(x),\dots,\beta_{U-2}(x) \in \{-1,+1\}$.  The claim follows.  

Conversely, if $(\alpha_1,\dots,\alpha_U)$ has an extension obeying $P^*_\Omega$, then for any $x \in G$, $\epsilon_1 \alpha_1(x)+\dots+ \epsilon_U \alpha_U(x)$ has magnitude at most $U-2$, and hence $P_\Omega$ holds.  This gives the desired claim.  
\end{proof}

Now we express the property of being a periodized permutation.

\begin{lemma}[Periodized permutations are expressible]\label{permutation-express}  Let $U \geq 1$, let $\Z/L\Z$ be a cyclic group with $L > 2U+4$, and let $1 \leq q \leq 2^{U-1}$.  Let $$\iota \colon \Z/q\Z \to \{-1,1\}^{U}$$ be an injection such that $\iota(\Z/q\Z)$ and $-\iota(\Z/q\Z)$ are disjoint (this is possible since $q \leq 2^{U-1}$).  Define the $(\Z \times \Z/2\Z, (\Z/L\Z)_{1 \leq u \leq U})$-property $Q_{\iota}$ by declaring a tuple of functions $$\alpha_u \colon \Z \times \Z/2\Z \to \Z/L\Z,\quad  1\leq u \leq U$$ to obey $Q_{\iota}$ if they are $(0,1)$-Boolean, and there exists a permutation $\sigma \colon \Z/q\Z \to \Z/q\Z$ such that
\begin{equation}\label{eq:pp}
    (\alpha_1(n, t),\dots,\alpha_U(n, t))) = \pm \iota( \sigma( \pi_q(n) ) )
\end{equation} 
for all $(n,t) \in \Z \times \Z/2\Z$ (where we use $x = \pm y$ as shorthand for $x \in \{ y, -y\}$).  Then $Q_{\iota}$ is expressible (in a constructive fashion).
\end{lemma}

\begin{proof}  First suppose that $(\alpha_1,\dots,\alpha_U)$ obeys  $Q_\iota$. Then from the $(0,1)$-boolean property and \eqref{eq:pp} we see that for any $(n,t) \in \Z\times \Z/2\Z$,  the tuples
$$ (\alpha_1(n+i, t+j),\dots,\alpha_U(n+i, t+j)),\quad i=0,\dots,q-1;\; j=0,1$$
 are in the set $\iota(\Z/q\Z) \uplus -\iota(\Z/q\Z)$, with each element of this set being attained exactly once. Moreover, since there are exactly $2q$ such tuples they must cover the set $\iota(\Z/q\Z) \uplus -\iota(\Z/q\Z)$, thus we have the functional equation
\begin{equation}\label{funceq}
 \biguplus_{i=0}^{q-1} \biguplus_{j=0}^1 (\alpha_1(n+i, t+j),\dots,\alpha_U(n+i, t+j)) +\{0\} = 
\iota(\Z/q\Z) \uplus -\iota(\Z/q\Z)
\end{equation}
for all $(n,t) \in \Z \times \Z/2\Z$.

Conversely, suppose that $\alpha_1,\dots,\alpha_U$ are $(0,1)$-boolean and obey the equation \eqref{funceq} for all $(n,t) \in \Z \times \Z/2\Z$.  Replacing $n$ by $n+1$ in \eqref{funceq}, setting $t=0$, and comparing the two resulting equations, we see that
$$ \biguplus_{j=0}^1 (\alpha_1(n, j),\dots,\alpha_U(n, j)) = \biguplus_{j=0}^1 (\alpha_1(n+q, j),\dots,\alpha_U(n+q, j)).$$
By the $(0,1$)-boolean nature of the $\alpha_u$, this implies that the set 
$$\{ (\alpha_1(n, 0),\dots,\alpha_U(n, 0)), -(\alpha_1(n, 0),\dots,\alpha_U(n, 0)) \}$$ 
is periodic in $n$ with period $q$.  Also, by \eqref{funceq} we have that the union $$\biguplus_{i=0}^{q-1}  \{ (\alpha_1(n, 0),\dots,\alpha_U(n, 0)), -(\alpha_1(n, 0),\dots,\alpha_U(n, 0)) \}$$ covers the set $\iota(\Z/q\Z) \uplus -\iota(\Z/q\Z)$.  We conclude that there exists a permutation $\sigma \colon \Z/q\Z \to \Z/q\Z$ such that
$$\{ (\alpha_1(n, 0),\dots,\alpha_U(n, 0)), -(\alpha_1(n, 0),\dots,\alpha_U(n, 0)) \} = \{ \iota(\sigma(\pi_q(n))), - \iota(\sigma(\pi_q(n))) \}$$
for all $n$.  This implies that the property $Q_\iota$ holds.  Since the property of $\alpha_1,\dots,\alpha_U$ simultaneously being $(0,1)$-boolean was already known to be expressible by Lemma \ref{bool-function} and Lemma \ref{closure}(i),(ii) (the latter being needed to combine together the separate assertions that each individual $\alpha_i$ is $(0,1)$-boolean), the claim then follows by a further application of Lemma \ref{closure}(ii).
\end{proof}

\subsection{Programming a Sudoku puzzle}\label{program-sec}

Now we can prove Theorem \ref{sudoku-to-expressible}, using a variant of the construction at the end of \cite[Section 7]{GT22}.  Let $\cS = \cS_{N,\Sigma}$ be a Sudoku rule, and let $\cC = \cC_q$ be an initial condition.  We let $s_0$ be a natural number obeying the largeness condition
$$ \# \Sigma, q \leq 2^{s_0-1}.$$
Then we can find injections  $\iota_0 \colon \Z/q\Z \to \{-1,1\}^{s_0}$, $\iota_1 \colon \Sigma \to \{-1,1\}^{s_0}$ such that  $\iota_0(\Z/q\Z)$ and $-\iota_0(\Z/q\Z)$ are disjoint, and such that $\iota_1(\Sigma)$ and $-\iota_1(\Sigma)$ are disjoint.  The maps $\iota_0, \iota_1$ can be viewed as an encoding of the sets $\Sigma, \Z/q\Z$ as binary strings.

We define $\Omega \subset \{-1,+1\}^{2s_0 N}$ to be the set of all tuples $$(\omega_{a,b,n})_{a = 0,1; b=1,\dots,s_0; n=1,\dots,N}$$ with $\omega_{a,b,n} \in \{-1,+1\}^{s_0}$, such that there exist a function $$g \colon \{0,\dots,N-1\} \to \Sigma$$ in $\cS$ and elements $c_1,\dots,c_N$ of $\Z/q\Z$ such that
$$ (\omega_{1,1,n},\dots,\omega_{1,s_0,n}) = \iota_1(g(n))$$
and
$$ (\omega_{0,1,n},\dots,\omega_{0,s_0,n}) = \iota_0(c_n)$$
and
$$ (c_n, g(n)) \in \cC$$
for all $n=1,\dots,N$.  Informally, $\Omega$ encodes the Sudoku rule $\cS$ and the initial condition $\cC$ in a binary form.  From the construction we see that $\Omega$ and $-\Omega$ are disjoint.

Let $\Z/L\Z$ be a cyclic group with $L > 4s_0 N + 4$.  Define the $$(\Z^2 \times \Z/2\Z, (\Z/L\Z)_{u=1,\dots,2s_0 N})\text{-property } S$$ by requiring a tuple $$(\alpha_{a,b,n})_{a=0,1; b=1,\dots,s_0; n=1,\dots,N}$$ of functions $\alpha_{a,b,n} \colon \Z^2 \times \Z/2\Z \to \Z/L\Z$ to obey the following properties:
\begin{itemize}
    \item[(1)]  For each $a=0,1$, $b=1,\dots,s_0$, and $n=1,\dots,N$, the function $\alpha_{a,b,n}$ is a $((0,0),1)$-boolean function.
    \item[(2)]  For each $a=0,1$, $b=1,\dots,s_0$, and $n=1,\dots,N$, the function $\alpha_{a,b,n}$ is $((-n,1),0)$-periodic.
    \item[(3)]  For each $((i,j),t) \in \Z^2 \times \Z/2\Z$, the tuple $$(\alpha_{a,b,n}((i,j),t))_{a=0,1; b=1,\dots,s_0; n=1,\dots,N})$$ lies in $\Omega \uplus - \Omega$.
    \item[(4)]  For each $n=1,\dots,N$, there exists $\sigma_n \colon \Z/q\Z \to \Z/q\Z$ such that
$$ (\alpha_{0,1,n}((i,j),t),\dots,\alpha_{0,s_0,n}((i,j),t)) = \pm \iota_0( \sigma_n( \pi_q(jn+i) ) )$$
for all $((i,j),t)  \in \Z^2 \times \Z/2\Z$.
\end{itemize}

From Lemma \ref{bool-function} and Lemma \ref{closure}(i),(ii) (the latter being needed to combine the assertions for each individual $\alpha_{a,b,n}$ together) we see that the property (1a) is expressible.  From Lemma \ref{express-period} and Lemma \ref{closure}(i),(ii) we similarly see that the property (2) is expressible.  From Lemma \ref{boolean-express} and Lemma \ref{closure}(i),(ii) we see that the property (3) is weakly expressible.  From Lemma \ref{permutation-express} and Lemma \ref{closure}(i),(ii) we see that the the conjunction of property (1) and property (4) is expressible.  By one final application of Lemma \ref{closure}(ii) we conclude that the property $S$ is also weakly expressible. Also it is clear that $S$ can be constructed in finite time from $\cS$ and $\cC$, and that the weak expressibility of $S$ is similarly constructive.

To complete the proof of Theorem \ref{sudoku-to-expressible}, it suffices to show that the property $S$ is satisfiable if and only if the $(\cS,\cC)$-Sudoku puzzle is solvable.

First suppose that the $(\cS,\cC)$-Sudoku puzzle is solvable, thus there exist a $\cS$-Sudoku solution $F \colon \mathbb{B} \to \Sigma$ and permutations $\sigma_n \colon \Z/q\Z \to \Z/q\Z$ such that
\begin{equation}\label{snf}
 ( \sigma_n(\pi_q(m)), F(n,m) ) \in \cC
\end{equation} 
for all $(n,m) \in \mathbb{B}$.  We then define the functions $\alpha_{a,b,n} \colon \Z^2 \times \Z/2\Z \to \Z/L\Z$
for $a=0,1$, $b=1,\dots,s_0$, and $n=1,\dots,N$, by the formulae
$$ (\alpha_{1,1,n}(i,j),t),\dots,\alpha_{1,s_0,n}(i,j),t)) \coloneqq (-1)^t \iota_1(F(n, jn+i))$$
and
$$ (\alpha_{0,1,n}((i,j),t),\dots,\alpha_{0,s_0,n}((i,j),t)) \coloneqq (-1)^t \iota_0(\sigma_n( \pi_q(jn+i))$$
for all $((i,j),t) \in \Z^2 \times \Z/2\Z$.  By construction we see that the functions $\alpha_{a,b,n}$ obey all of the above properties (a)-(d), and hence obey the property $S$.  Thus $S$ is satisfiable.

Conversely, suppose that $\alpha_{a,b,n} \colon \Z^2 \times \Z/2\Z \to \Z/L\Z$ for $a=0,1$, $b=0,\dots,s_0$, and $n=1,\dots,N$ are a collection of functions obeying property $S$.  From properties (1), (2), (3) and the construction of $\Omega$, we see that there exist functions $F \colon \mathbb{B} \to \cW$ and $c \colon \mathbb{B} \to \Z/q\Z$ such that
$$ (\alpha_{1,1,n}((i,j),t),\dots,\alpha_{1,s_0,n}((i,j),t)) = \pm \iota_1(F(n, jn+i))$$
and
$$ (\alpha_{0,1,n}((i,j),t),\dots,\alpha_{0,s_0,n}((i,j),t)) = \pm \iota_0(c(n, jn+i))$$
for all $((i,j),t) \in \Z^2 \times \Z/2\Z$ and $n=1,\dots,N$, such that $n \mapsto F(n,jn+i)$ lies in $\cS$
and
$$ (c(n,jn+i), F(n,jn+i)) \in \cC$$
for all $n  \in \{0,\dots,N-1\}$ and $i,j \in \Z$.  Thus $F$ is a $\cS$-Sudoku solution.   From property (4), the injectivity of $\iota_0$, and the disjointness of $\iota_0(\Z/q\Z)$ and $-\iota_0(\Z/q\Z)$, we have
$$ c(n,jn+i) =  \sigma_n( \pi_q(jn+i) )$$
for all $n  \in \{0,\dots,N-1\}$ and $i,j \in \Z$, where $\sigma_n$ are the permutations in property (4), thus \eqref{snf} holds for all $(n,m) \in \mathbb{B}$. We conclude that $F$ obeys the initial condition $\cC$, and so the $(\cS, \cC)$-Sudoku puzzle is solvable.  This concludes the proof of Theorem \ref{sudoku-to-expressible}.

\section{Further discussion and open problems}

\subsection{Possible improvements}
We suggest several possible improvements of our construction.

\begin{itemize}
    \item It is of interest to remove the role of the periodic set $E$ in the statement of Theorem \ref{thm:main}, namely lifting the periodic subset $E$  to be the whole group $E=G_0$. See Remark \ref{rem:persub} for further discussion.
    \item  In our undecidability statement of Theorem \ref{thm:main}  the group $G_0$  is not fixed but rather given as part of the input.  It is of interest  to know if there is a fixed finite Abelian group $G_0$ such that translational monotilings in $\Z^2\times G_0$ are algorithmically undecidable. 
    \begin{question}\label{Q:fixedG0}
        Does there exist a finite Abelian group $G_0$ such that translational monotilings in $\Z^2\times G_0$ are algorithmically undecidable?
    \end{question}
    In \cite{GT21} we showed the undecidability of monotilings in spaces of the form $\Z^2\times G_0$ where $G_0$ is a finite non-Abelian group. The undecidability of monotilings in a fixed group of this form is still open. Since our tiling language is significantly more expressible when $G_0$ is non-Abelian, we expect that solving this question would be significantly easier than Question \ref{Q:fixedG0}.
\end{itemize}

\subsection{Decidability of periodic translational tilings}    
In \cite{GK72}, using Berger's construction \cite{Ber, Ber-thesis}, it was shown that periodic translational tilings with multiple tiles are undecidable in $\Z^2$: there is no algorithm that, when given a tile-set in $\Z^2$, computes in finite time whether it admits a strongly periodic tiling of $\Z^2$. The decidability of periodic translational monotilings in $\Z^2$ follows from \cite{BH, GT20}, where it was shown that a finite set $F$ in $\Z^2$ is a translational monotile if and only if it admits a strongly periodic tiling.

Although in this paper we establish the undecidability of translational monotilings in virtually $\Z^2$ spaces, our argument leaves open the decidability of periodic translational monotilings in virtually $\Z^2$ spaces, since all the translational monotilings that arise from our encoding of $(\cS,\cC)$-Sudoku puzzles are aperiodic. 

\begin{question}
    Does there exist an algorithm that, when given a translational monotile in a virtually $\Z^2$ space, computes in finite time whether it admits a strongly periodic tiling? 
\end{question}

\subsection{Decidability of non-translational tilings in the plane}
The periodic tiling conjecture is known to hold for topological disks \cite{bn, ken}.\footnote{We hope to extend this result beyond topological disks in a future work (in preparation).} 
This implies a decidability result for translational monotilings by topological disks. However, if allowing a larger group of motions to act on the monotiles, there are several recent constructions of planar aperiodic monotilings \cite{socolar-taylor, hat, spectre}. The constructions in \cite{hat, spectre} consist of monotiles whose shape is very simple; in particular, all of them are topological disks. Nevertheless, the decidability of such monotilings is still unsolved. See \cite[Section 7]{hat} for further discussion.

\subsection{Translational tilings by multiple tiles in $\Z^2$}
Theorem \ref{thm:main} gives the undecidability of translational monotilings in virtually $\Z^2$ spaces. In $\Z^2$, translational monotilings are known to be decidable \cite{BH, GT20}. In \cite{ollinger11} it was shown that  translational tilings with $11$ tiles in $\Z^2$ are undecidable. It is still an open problem to determine the minimal $2\leq J\leq 11$ such that translational tilings with $J$ tiles in $\Z^2$ are undecidable.

\end{document}